\documentclass[12pt]{article}
\usepackage[english]{babel}
\usepackage{amsmath}
\usepackage{amssymb}
\usepackage{amsthm}
\usepackage{thmtools}
\usepackage{amsfonts}
\usepackage{float}
\usepackage{aliascnt}
\usepackage{fancyhdr}
\usepackage{cleveref}
\usepackage{graphicx}
\usepackage{soul}
\usepackage{xcolor}
\usepackage{bbm}
\usepackage{url}
\usepackage{stmaryrd}
\usepackage[margin=1in]{geometry}
\usepackage[utf8]{inputenc}
\usepackage[english]{babel}
\makeatletter
\def\thmt@refnamewithcomma #1#2#3,#4,#5\@nil{%
  \@xa\def\csname\thmt@envname #1utorefname\endcsname{#3}%
  \ifcsname #2refname\endcsname
    \csname #2refname\expandafter\endcsname\expandafter{\thmt@envname}{#3}{#4}%
  \fi
}
\newcommand{\indep}{\rotatebox[origin=c]{90}{$\models$}}
\makeatother

\usepackage{etoolbox}
\title{}
\newcommand{\N}{\mathbb{N}}
\newcommand{\con}{\rightarrow}
\newcommand{\conx}[1]{\stackrel{#1}{\rightarrow}}
\newcommand{\R}{\mathbb{R}}
\newcommand{\ve}{\varepsilon}
\newcommand{\define}[4]{\expandafter#1\csname#3#4\endcsname{#2{#4}}}

\newcommand{\ccK}{{\cal K}}
\newcommand{\ql}{\mathbbm{1}}

\newcommand{\lp}{\left( }
\newcommand{\rp}{\right) }
\newcommand{\lb}{\left\lbrace}
\newcommand{\rb}{\right\rbrace}
\newcommand{\lf}{\left[}
\newcommand{\rf}{\right]}

\newcommand{\lv}{\left\vert}
\newcommand{\rv}{\right\vert}
\newcommand{\lV}{\left\Vert}
\newcommand{\rV}{\right\Vert}
\newcommand{\E}{\mathbb E}

\declaretheoremstyle[
  spaceabove=\topsep,
  spacebelow=\topsep,
  headfont=\normalfont\bfseries,
  notefont=\bfseries, notebraces={(}{)},
  bodyfont=\normalfont\itshape,
headfont=\normalfont\bfseries,
notefont=\bfseries, notebraces={(}{)},
bodyfont=\normalfont,
postheadspace=\newline,
qed=$\circ$
]{exstyle}

\declaretheorem[style=bombom,title=Theorem,numberwithin=section,refname={theorem}]{s}
\declaretheorem[style=bombom,name=Lemma,sibling=s,refname={lemma}]{lm}
\declaretheorem[style=bombom,name=Definition,sibling=s,refname={definition}]{defi}
\declaretheorem[style=bombom,name=Corollary,sibling=s,refname={corollary}]{koro}
\declaretheorem[style=exstyle,name=Example,sibling=s,refname={example}]{eks}
\declaretheorem[style=bombom,name=Proposition,sibling=s,refname={proposition}]{prop}

\declaretheoremstyle[
  spaceabove=\topsep,
  spacebelow=30pt,
  headfont=\normalfont\bfseries,
  notefont=\mdseries, notebraces={(}{)},
  bodyfont=\normalfont,
 postheadspace=\newline,
 numbered=no,
  qed={\mathbb E}dsymbol,
  name=Proof
]{mythmstyle}

\newtheorem{rem}{Remark}
\newtheorem{ass}{Assumption}

\title{Stability and Mean-Field Limits of Age Dependent Hawkes Processes}
\author{Mads Bonde Raad \thanks{Department of Mathematical Sciences,
		University of Copenhagen, Universitetsparken 5, DK-2100 Copenhagen}\\
\and
Susanne Ditlevsen \footnotemark[1]
	\and
	Eva L{\"o}cherbach\thanks{
		Universit{\'e} de Paris 1 Panth\'eon-Sorbonne, SAMM, F-75013 Paris}
}

\begin{document}
\maketitle

\begin{abstract}
	In the last decade, Hawkes processes have received a lot of
        attention as good models for functional connectivity in neural
        spiking networks. In this paper we consider a variant of this
        process; the age dependent Hawkes process, which incorporates
        individual post-jump behavior into the framework
        of the usual Hawkes model. This allows to model recovery
        properties such as refractory
        periods, where the effects of the network are momentarily being
        suppressed or altered. We show how classical stability results for Hawkes
        processes can be improved by introducing age into the
        system. In particular, we neither need to a priori bound the
        intensities nor to impose any conditions on the Lipschitz
        constants. When the interactions between neurons are of
        mean-field type, we study large network limits and establish
        the propagation of chaos property of the system.
        
         \begin{center}
        	\textbf{R\'esum\'e}
        \end{center}
        Depuis la derni\`ere d\'ecennie il s'est av\'er\'e que la classe des processus de Hawkes fournit un bon mod\`ele pour d\'ecrire la connectivit\'e fonctionnelle dans un r\'eseau de neurones. Dans cet article nous \'etudions une variante de ce processus, le processus de Hawkes structur\'e en \^age. Cette structure en \^age rajoute un comportement individuel apr\`es les sauts \`a la dynamique de chaque composante, ce qui permet en particulier de d\'ecrire une p\'eriode refractaire durant laquelle l'influence du r\'eseau est supprim\'ee ou au moins modifi\'ee. Nous am\'eliorons les r\'esultats de stabilit\'e classiques pour les processus de Hawkes dans ce cadre. En particulier, nous n'avons ni besoin de supposer que les intensit\'es sont born\'ees, ni d'imposer une condition aux normes Lipschitz des fonctions taux de saut. Lorsque les interactions entre les neurones sont du type champ moyen, nous \'etudions les limites en grande population et nous d\'emontrons la propri\'et\'e de propagation du chaos du syst\`eme.   
\end{abstract}



\section{Introduction}
In nature, macroscopic events caused by many microscopic events in an
interacting network of units often
exhibit a cascading structure, so
that they come in waves, for example caused by some events inhibiting
or exciting the occurrence
of other events. Technological development permitting high-frequency
data sampling in recent decades has made it possible to perform
detailed analysis of the
dependencies between units interacting in a cascade structure. It is therefore
relevant to develop models which can quantify temporal
interactions between many units on a microscopic level. There are many
examples of phenomena of interest which occur in cascades and have
been analyzed by Hawkes processes.
Examples include bankruptcies in
finance that propagate through a market, giving rise to volatility clustering
observations \cite{Hawkes-In-Finance},
interactions on social media \cite{Social-Media}, and pattern dependencies in DNA \cite{pat}.

Hawkes processes are point processes where
the intensity function is stochastic and allowed to depend on the past
history, introducing memory in the temporal evolution of the
stochastic process. They have commonly been
used to model neurophysiological processes
\cite{chevallier,ccdr,SusEva,GerhardDegerTruccolo2017,patvin}, and the
application we have in mind is to model functional connectivity
between neurons in a network. When neurons send an electric signal,
the so-called {\em action potential} or {\it spike}, they excite or inhibit recipient neurons in
the network (the {\em post-synaptic} neurons).  Jumps of the $i$th
unit of the Hawkes
process are then identified with the spike times of the $i$th
neuron. Moreover, the biological process imposes a strong
self-inhibition on a neuron that has just emitted a spike. This period
of about $2 ms$ is called the \textit{absolute refractory period}, and
in this phase it is virtually impossible for a neuron to spike
again. The neuron then gradually regains its ability to spike in the
longer \textit{relative refractory period}. It was proposed in
\cite{chevallier} to model absolute and relative refractory periods in
neuronal spike trains by {\it age dependent Hawkes processes}, where
the age of a unit is defined as the time passed since the last time it
jumped, and thus, it resets to zero at each jump time. The
present article is devoted to a thorough study of its stability
properties and associated mean-field limits.

It turns out that the classical stability results for Hawkes
        processes can be improved by introducing age into the
        system. In particular, we neither need to a priori bound the
        intensities nor to impose any conditions on the Lipschitz
        constants. This is interesting not only from a mathematical
        but also from a biological point of
        view. If a network of neurons is transmitting some
        information over time, it is not operational nor realistic that
        intensities should be bounded, a given signal should be
        transmitted without any necessary delay. However, if the
        activity explodes, the entire system breaks down. Introducing
        a refractory period on the individual neuron stabilizes the
        system, while at the same time the information is still
        transferred effectively by the network.

In the present paper we consider multivariate counting processes $(Z_t^i)_{t \geq 0}, i = 1, \ldots , N$, where
$N$ is the number of units in the network.

The counting processes $(Z_t^i)_{t \geq 0}, i = 1, \ldots , N, $ are characterized by their \textit{conditional intensities} which can informally be described as the instantaneous jump rate, given the past,
that is,
$$  \lambda_t^i  dt \approx P ( Z^{i}\mbox{ has a jump in } \left( t ,
  t + dt \right]\vert {\mathcal F}_{t}),$$
 where ${\mathcal F}_{t}$ is the history of the entire network of neurons.
 The Hawkes process is defined by imposing a specific structure on the conditional intensity.

Before giving a precise definition of the age dependent Hawkes process in \Cref{defiHawkes} below,
let us start by describing and discussing its related components in a
less formal way. We consider an $N-$dimensional point process $\lp
Z^{i}\rp_{i\leq N}$, where each  coordinate $Z^{i}$ counts the jump
events of the $i$th unit. The intensity of this process is a
predictable process depending on the history ${\mathcal F}_{t}$ before
and up to time $t.$ It is assumed to have the form
\begin{equation}
\lambda^{i}_{t}=\psi^{i}\lp X^{i}_{t} ,A^{i}_{t}\rp,
\end{equation}
where  $X^{i}_{t}$ is the \textit{memory process}, a predictable
  process depending on the history  ${\mathcal F}_{t}$ of the
  process,  $\psi^{i}$ is the
\textit{rate function,} and $A_t^{i} $ is the {\it age process} of
$Z^{i}$. Here follows a brief introduction of the involved objects.

\textbf{The Rate Function $\lp x,a\rp\mapsto \psi^{i}\lp x,a\rp$}
describes how the memory and the age influence the intensity of the $i$th unit. The existence of a
non-exploding Hawkes process is generally ensured by assuming that
$\psi^{i}$ is sub-linear in $x$. Often, stronger assumptions such as a
uniform bound on $\psi^{i}$ is also imposed to prove basic
properties. In this article we will work under standard Lipschitz- and
linear-growth-conditions, but we shall not need to bound the rate
function nor its Lipschitz constant.

\textbf{The Age Process $A^{i}_t$} associated to the $i$th process $Z^{i}$ is the
time elapsed since the last jump time of $Z^{i}$ before time $t$, that is,
\begin{align*}
A^{i}_{t} =\begin{cases}
 A^{i}_{0}+t, & \mbox{if  } Z^{i} \text{ has not jumped between time 0
   and time }t,\\
 t-\sup\lb s<t : \Delta Z^{i}_{s}>0\rb ,   & \text{otherwise,}
\end{cases}
\end{align*}
where $\Delta Z^{i}_{s} =
Z^{i}_{s}-Z^{i}_{s-}=Z^{i}_{s}-\lim_{\varepsilon \rightarrow
  0+}Z^{i}_{s-\varepsilon}$ are the jumps.

\textbf{The Memory Process} $X^{i}_t$ integrates the effects of previous
jumps in the network, where the influence from the past is a weighted average of all
previous jumps of all units that directly affect unit $i$ (the {\em
  pre-synaptic} neurons). Each unit has its own memory process, even
if they all depend on the same common history of all units, but they
are affected in individual ways. More precisely, the $i$th memory process is
assumed to have the structure
\begin{align*}
X^{i}_{t} = \sum_{j=1}^{N}\sum_{\tau < t : \Delta Z^{j}_\tau = 1 } h_{ij}\lp t-\tau\rp +R^{i}_{t}=\sum_{j=1}^{N}\int_{0}^{t-}h_{ij}\lp t-s\rp Z^{j}(ds) + R^{i}_{t}.
\end{align*}
In the definition of the memory process  we have introduced two new
objects.

\indent  \textbf{The Weight Function} $h_{ij}\lp t\rp$ determines how much a jump of unit $j$ that occurred $t$ time
units ago contributes to the present memory of unit $i$. Positive
$h_{ij}(t)$ means excitation of unit $i $ when a jump of unit $ j $
occurred $t$ time units ago, while
negative $h_{ij}(t)$ means inhibition.

\textbf{The Initial
  Signal} $R^{i}_t$ is a process assumed to be known at time $t=0.$ It
should be thought of as a memory process which the process inherits
from past time.

When $R^{i}\equiv 0$ and $\psi^{i}\lp
x,a\rp=f^{i}\lp x\rp$ for a suitable function $f^{i},$ we obtain the
usual non-linear Hawkes process which has been studied in detail,
e.g., in \cite{bm}.

In Section \ref{sec:stability}, we discuss stability of the age dependent Hawkes process.
The main assumption is a post-jump bound on the intensity,
corresponding to a strong self-inhibition for a short time interval
after a spike. This models the refractory
period. We do not impose any a priori bounds on  the intensities. Within this sub-model, we are able to prove stability
properties for the $N-$dimensional Hawkes process. The results we obtain are
similar to what has been shown for ordinary nonlinear Hawkes processes
in \cite{bm}   and recently in \cite{costa}. This last paper is however entirely devoted to the study of
weight functions which are of compact support giving rise to explicit regeneration points when the process comes back to the all zero measure.
Compared to these studies, it turns out that the natural self-inhibition by the
age processes eliminates the need of controlling the Lipschitz
constant of $\psi,$ and we do not need any restriction on the support
of the weight functions. We also discuss which starting conditions (that
is, which form of an initial process)  will ensure that the Hawkes process jumps in synchrony with the invariant process eventually. When this holds, the Hawkes process is said to \textit{couple} with the invariant process. These results are collected within our first main
theorem, Theorem \ref{Stab}.

During the proof of the stability properties, other
interesting properties of the model are discussed, such as a
nice-behaving domination of the intensities (Lemma \ref{Xbound}).

In Section \ref{sec:mf}, we study a mean-field setup and associated
mean-field limits. More precisely, we consider $N$ interacting units
which are organized within $\ccK$ classes of
populations. Each unit  belongs to one of these classes, and any two units within the
same class $k$ are assumed to be similar, $k=1,\dots ,\ccK$. This means that they have
the same rate function $\psi^{k}$,  memory process
  $X^k$ and initial signal $R^{k}$, and the
weight function describing the influence of any unit belonging to
another class $l$ is given by $N^{-1}h_{kl}$. However, each unit
still has its own age process.

In this setup we establish a limiting distribution for a large scale
network, $N\con \infty$. Let $N_{k}$ be the number of units in class
$k$ with proportion $N_{k}/N$, and assume $\lim_{N\rightarrow \infty}
N_k/N = p_k > 0$. We index the $j$th unit within class
$k$ by $Z^{kj}$, $k=1, \ldots, \ccK, j=1, \ldots , N_{k}$. It is
sensible to assume that small contributions from unit $Z^{lj}$ to the memory process
$X^{k}$ of the $k$th class disappear in the
large-scale dynamics, meaning that
$N^{-1}\sum_{j=1}^{N_{l}} \int_0^{t-} h_{kl} (t-s) Z^{lj} (ds ) \approx p_{l} \int_0^t h_{kl} (t-s)d {\mathbb E}
Z^{l1} (s) $ for large $N$,  for any $ 1 \le l \le \ccK.$ Therefore, if a limiting point process $\lp
Z^{kj}\rp_{j\in \N}$ exists, we expect that any $Z^{kj}, k \le \ccK, j \geq 1, $ should have
intensity
\begin{align*}
\lambda^{kj}_{t} = \psi^k\lp x^{k}_{t},A^{kj}_{t}\rp ,
\end{align*}
where $A^{kj}$ is the age process of $Z^{kj}$, and the process $x^{k}_{t}$ is deterministic, given by
\begin{align*}
x^{k}_{t} = \sum_{l=1}^{\ccK}p_{l}\int_{0}^{t} h_{kl}\lp t-s\rp d {\mathbb E} Z^{l1} (s) +r^{k}_{t},
\end{align*}
where $r_t^k$ is a suitable limit of the initial processes in
population $k$.
In Theorem \ref{koroet}, we discuss criteria under which such a system
exists. Our second main theorem, Theorem \ref{theo:prop}, shows that
this system will indeed be a limit process for the age dependent Hawkes
processes for $N\con \infty$. We also discuss in Lemma
\ref{weightapproximation} how robust the system is to adjusting the
weight functions. Not only is this robustness a good  model feature
in itself, but it also allows approximation of an arbitrary age
dependent Hawkes process, using weight functions with better
features. Examples are weight functions given by Erlang densities or
exponential polynomials which induce Markovian systems, see \cite{SusEva}.

We close our article with an Appendix where we collect some proofs and
useful results about counting processes.

\subsection*{Notation, Definitions and Core Assumptions}
Throughout this article, we will be working on a background probability space $\lp \Omega, {\mathcal F},P\rp ,$ and  all random variables are assumed to be defined on this space. If $v= (v_1, \ldots , v_d ) $ is a $d-$dimensional Euclidian vector, then $\lv v\rv = \sum_{i=1}^d \lv v_{i}\rv$ denotes the $1$-distance.  Moreover, for a $d-$dimensional process $X$ we define the running-supremum of the $1$-distance as $\|  X \|_{t}=\sup_{s\leq t}\lv X_{s}  \rv$.

We recall the basic Stieltjes integration notation. Let $g: \R\con \R$ be a càdlàg function. The variation of $g$ on a bounded interval $I$ is given by
	\begin{align*}
	V_{g}\lp I\rp  =\sup_{\lp x_{i}\rp \in \mathcal{I}} \lv g\lp x_{i}\rp-g\lp x_{i-1}\rp\rv<\infty.
	\end{align*}
	where $\mathcal{I}$ denotes the system of all finite sets of inceasing indices $\lp x_{i}\rp\subset I.$
	$g$ is said to be of finite variation, if the variation is finite on all bounded intervals. For such $g$ there exist two singular $\sigma-$finite measures $\mu_{g}^{+},\mu_{g}^{-}$ s.t. $\mu_{g}:=\mu_{g}^{+}-\mu_{g}^{-}$ which satisfies $\mu_{g}\lp a,b\rf=g\lp b\rp-g\lp a\rp.$ The variation measure $\lv dg\rv:=\mu_{g}^{+}+\mu_{g}^{-}$ satisfies $\lv dg\rv\lp a,b\rf= V_{g}\lp a,b\rf$.  If $f: \R\con \R $ is a measurable function such that  $\int \lv f\rv  \lv dg\rv<\infty$ then we define the \textit{Lebesgue-Stieltjes integral} as
	\begin{align*}
	\int f\lp x\rp dg\lp x\rp = \int f\lp x\rp d\mu_{g}\lp x\rp,
	\end{align*}
	see e.g. \cite{Halmos}. If $\nu$ is a measure on $\R^{2}$ we shall also use the following notation for the integral over semi-closed boxes
	\begin{align*}
	\int_{a}^{b}\int_{c}^{d} f\lp x,y\rp \;\nu\lp dx,dy\rp = \int \ql \lb \lp a,b\rf \rb \lp x\rp \ql \lb \lp c,d\rf \rb \lp y\rp f\lp x,y\rp \;\nu\lp dx,dy\rp.
	\end{align*}

In the following we introduce the core mathematical objects and
assumptions needed to discuss the age dependent Hawkes process.
\begin{itemize}
	\item[$\pi\; \vert $] $\pi $ and $\pi^{i}, i \in \N , $ are i.i.d.\ Poisson Random Measures (PRMs)  on $\R \times \R_{+} $ with Lebesgue intensity measure.
For any  $t \in \R ,$ we define the  $\sigma$-algebra $\tilde {\mathcal F}_{t}$ induced by the projections
	\begin{align*}
	\pi\lp A\cap \lp \lp -\infty,t \rf\times \R_{+} \rp \rp, \;
          \pi^{i} \lp A\cap \lp \lp -\infty,t \rf\times \R_{+} \rp
          \rp\mbox{ for }   A\in {\mathcal B}\lp \R\times \R_{+}\rp,
          \, i \in \N .
	\end{align*}
	We equip the space $\lp \Omega, {\mathcal F},P\rp$ with the
        filtration $( {\mathcal F}_t)_{t \in \R } $ which is the
        completion of  $\lp \tilde {\mathcal F}_{t} \rp_{t\in \R}.$

	\item[$h\; \vert $] weight functions: For all $ 1 \le i , j \le N, $ $h_{ij}  :\R_{+}\rightarrow \R$  is a  locally integrable function.
	\item[$R_{t}\; \vert  $]  initial signals: For all $ 1 \le i \le N, $  $(R^i_{t})_{t \geq 0}$ is an  $\mathcal{F}_{0 }\otimes \mathcal{B} $ measurable process on $t\in \R_{+}$ such that  $t \mapsto  {\mathbb E} R^i_{t}$ is locally bounded.
	\item[$\psi\; \vert  $]   rate functions: For all $1 \le i \le N, $   $\psi^i : \R\times \R_{+}\con \R_{+}$ is a measurable function which is $L$-Lipschitz in $x$ when the age variables agree,  and otherwise sub-linear in $x,$ i.e.,
	\begin{align}\label{psias1}
	\lv \psi^i\lp x,a\rp-\psi^i\lp x',a'\rp\rv\leq L \lv x-x' \rv\ql\lb a=a'\rb +L\lp \max\lp \lv x'\rv,\lv x\rv \rp +1 \rp \ql\lb a\neq a' \rb ,
	\end{align}
	for some $L \geq 1. $ By taking a possibly larger $L$ we may, and will, also assume that
	\begin{align}\label{psias2}
	\psi^i\lp x,a\rp\leq L\lp 1+|x| \rp\quad \forall \lp x,a\rp \in \R\times \R_{+} .
	\end{align}
	\item[$A_{0}\; \vert $]  initial ages $\lp A^{i}_{0} \rp_{i\in \N}$ are $ {\mathcal F}_0-$measurable random variables with support in $\R_{+}$.
\end{itemize}

With these definitions, we introduce the age dependent Hawkes process.
	 \begin{defi}[The age dependent Hawkes process]
	\label{defiHawkes}
	Let $N\in \N$ and let $\lp Z,X,A\rp = ( (Z^i)_{1 \le i \le N},  (X^i )_{1 \le i \le N }, (A^i )_{1 \le i \le N } ) $ be a triple consisting of an $N-$dimensional counting process $Z$, an $N-$dimensional predictable process $X$, and an $N-$dimensional adapted c\`agl\`ad process $A$. The triple is an $N-$dimensional age dependent Hawkes process with weight functions $\lp h_{ij} \rp_{i,j\leq N}$, spiking rates $\lp \psi^{i} \rp_{i\leq N}$, initial ages $\lp A^{i}_{0}\rp_{i\leq N}$, and initial signals $\lp R^{i} \rp_{i\leq N}$ if almost surely all sample paths solve the system
	\begin{eqnarray}
	Z_t^{i} &= &\int_{0}^{t}\int_{0}^{\infty}\ql\lb z\leq \psi^{i}\lp X^{i}_{s},A^{i}_{s}\rp\rb  \pi^{i}\lp ds,dz\rp, \nonumber \\
	X^{i}_{t}&=&\sum_{j=1}^{N}\int_{0}^{t-} h_{ij}\lp t-s\rp Z^{j}(ds)+R^{i}_{t} , \label{system}\\
	A^{i}_{t}-A^{i}_{0} &=& t-\int_{0}^{t-} A^{i}_{s} \, Z^i (ds)    ,\nonumber
	\end{eqnarray}
for all $ t \geq 0.$
\end{defi}

\begin{rem}
		Notice that we choose the c\`agl\`ad version of $A$, that is, $A$ is left continuous and has right limits. All age processes presented in the article will be c\`agl\`ad as well. This is notationally convenient as the age process will appear in the intensities for most point processes treated in this article.
	\end{rem}

\begin{eks}[Examples of rate functions]\label{Ex:rate}
	A possible choice of rate function is $ \psi^i ( x, a ) = f^i
	(x) g^i (a) $,  where $f^i$ is  Lipschitz
	and  $g^i$ is bounded. In particular, for the neuroscience application, it is possible to model an absolute refractory period of length $\delta$ by putting $ \psi^i (
	x, a) = f^i (x) \ql \{ a > \delta\} .$
	Another example are the rate functions considered in the time elapsed neuron network model of \cite{salort} given by
	$ \psi^i ( x, a) = f^i (x) g^i ( a - s^i (x) ) ,$
	where $g^i \equiv 0 $ on $\R_-,$ and where $g^i, s^i $ are Lipschitz and bounded.
\end{eks}

\begin{eks}[Initial process]
Let $ z^i  (ds) , 1 \le i \le N, $ be ${\cal F}_0-$measurable point measures on $ [ -
\infty , 0 ] $ which we interpret as the initial condition of the age
dependent Hawkes process. We then typically think of  initial
processes $R_t^i$ of the form
$$ R_t^i = \sum_{j=1}^N \int_{- \infty }^0 h_{ij} (t-s)  z^j (ds) ,$$
provided the above expression is well-defined.
\end{eks}
Well-posedness of the system \eqref{system} follows from
\begin{prop}\label{prop:13}
	Almost surely, there is a unique sample path $\lp Z,X,A\rp$ solving \eqref{system}. Moreover, $Z$ is non-exploding.
\end{prop}
\begin{proof}
	 Let $h=\sum_{i,j=1}^{N}\lv h_{ij}\rv$ and { $R = \sum_{i = 1
           }^{N }\lv R^{i }\rv $. Assume first that $R $ is bounded by
           some constant $M>0 .$} Consider the linear Hawkes processes
	\begin{eqnarray}\label{eq:ztilde}
	\tilde{Z}^{i}_{t} &=& \int_{0}^{t}\int_{0}^{\infty}\ql\lb z\leq L\lp 1+Y_{s}\rp\rb \pi^{i}\lp ds,dz\rp ,  1 \le i \le N,  \nonumber \\
	Y_{t}&=& \sum_{j=1}^{N}\int_{0}^{t-}  h\lp t-s\rp  \tilde{Z}^{j}  (ds) +M.
	\end{eqnarray}
	Notice that $\tilde{Z}$ is driven by the same PRMs as $Z$. It is well known  that the system \eqref{eq:ztilde} almost surely has a path-wise unique solution, which is defined for all  $t \in \lf 0 ,\infty\rp $ (see for example \cite{dfh}, Theorem 6).  By induction over jump times of $\sum_{j=1}^{N}\tilde{Z}^{j}$ it follows that \eqref{system} has a unique solution satisfying $Z_t^{i}\leq \tilde{Z}_t^{i}$ for all $t\in \R_+,$ implying that $ Z$ does not explode.\\\\
   {In the general case define for each $m \in \N $ the $
     N-$dimensional age dependent Hawkes processes $ ( Z^m , X^m ,
     A^m ) $ with the
     same starting conditions and parameters as $\lp Z,X,A \rp $
     except for the initial processes, which are instead defined as  $R^{mi } = -m\vee R^i \wedge m. $ Write also $  \lambda_t^{ m , i} = \psi^i (  X^{m, i }_t,  A_t^{m, i } ) $ for the associated intensity and  $  \lambda_t^{m \wedge m+1, i } :=  \lambda^{m, i}_t \wedge  \lambda^{m+1, i }_t .$ We have that
  \begin{align*}
  P &(  Z_{t }^m \neq  Z_{t }^{m+1},\;t \in [0,T] )   \\
  &\le \, P \left( \exists i :\int_0^\infty \int_{0}^{T } \ql \{ z \in (  \lambda_s^{m \wedge m+1, i } ,  \lambda_s^{m \wedge m+1, i }  + L | R^{m, i}_s  - R^{m+1 , i}_s |  \} \pi^i \lp ds,dz\rp \geq  1  \right)  \\
  &\le \, L \sum_{i=1}^N   \E \int_0^T | R^{m, i}_s  - R^{m+1 , i}_s |
    \, ds  .
\end{align*}
  Since $ \E \int_0^T |R^i_s | ds < \infty , $ we conclude that
  $$ \sum_m P (  Z^m_{t } \neq  Z_{t }^{m+1},\;t \in  [0,T] ) < \infty , $$
  implying that almost surely, the limit $Z = \lim_{m\to \infty }Z^{m } $ exists. It is straightforward to show that $Z $ solves \eqref{eq:ztilde}.}

\end{proof}

\section{Stability}\label{sec:stability}
We start this section by discussing the stability of the age dependent Hawkes process
within a sub-model where age acts as an inhibitor. For nonlinear
Hawkes processes with no age dependence, a thorough investigation of
invariant distributions and couplings was done in \cite{bm}. However,
for results where boundedness is not forced upon the system (Theorem
1 of \cite{bm}), stability depends on the Lipschitz constant $L$ of
$\psi$ in \eqref{psias1}. As we show in \Cref{Stab} below, such
restrictions are not necessary when age is incorporated as an
inhibitor. The fact that the model has desirable stability properties
is indicated by the result of Lemma \ref{Xbound}, where we
state a strong control of the intensity.

Throughout this section, the processes are defined
on the entire real line $\R$ unless otherwise mentioned. We do this
for the following reason. When studying stability and thus the
existence of stationary versions of infinite memory processes such as
(age dependent) Hawkes processes, a widely used approach is to
construct the process starting from $t = - \infty .$ If such a
construction is feasible, this implicitly implies that the state of
the process at time $t=0$ must be in a stationary regime. Therefore,
throughout this section we will work with random measures $ Z $
defined on the entire real line, with the usual
identification of processes and random measures given by  $Z_t = Z (
(0, t ] ) ,$ for all $ t \geq 0,$ and $Z_t  = - Z ( (t, 0]), $ for all
$ t < 0 .$ We shall also use the shift operator $ \theta^r $ which is defined for any $ r \in \R$ by
\begin{equation}
\label{eq:shiftoperator}
\theta^r Z ( C ) := Z ( r + C ) := Z ( \{ r + x : x \in C \} ) ,
\end{equation}
for any $ C \in {\cal B} ( \R).$

\paragraph{Setup in this section:}
We consider a system with a fixed number of units $N$. Introduce the functions
$$ \overline{h}_{ij}( t) =\sup_{s \geq t} | h_{ij}| ( s),\ h\lp t \rp=\sum_{i,j = 1 }^{N }\lv h_{ij }\rv  .$$
In addition to the fundamental assumptions we add the following set of assumptions.
\begin{ass}\label{as1}
1. There exist $ K $ and $ \delta > 0 $ such that
\begin{equation}\label{gassump}
	 \psi^i \lp x,a\rp\leq K \text{ for all } 1 \le i \le N, a\in [0,\delta],x\in \R.
	\end{equation}
2. There exist
	$x^* , a^* , c > 0 $ such that for all $ |x| \le x^*, a \geq a^* $ and for
	all $ 1 \le i \le N, $
	\begin{equation}\label{eq:doeblin}
	\psi^i ( x, a ) \geq c > 0 .
	\end{equation}
3. We suppose that
	\begin{equation}\label{eq:hbar}
	[0, \infty [ \ni  t\mapsto  \overline{h}_{ij}( t) \in
	{\mathcal L}^{1} \cap {\cal L}^2   \, \mbox{ and  }  \, \;   [0, \infty [ \ni  t\mapsto t  h_{ij}( t) \in {\mathcal L}^1.
	\end{equation}
	\end{ass}
\noindent Notice that \eqref{eq:hbar} implies that
	\begin{equation}\label{eq:hbar2}
	\overline{h}:=\sum_{i,j=1}^{N} \overline{h}_{ij} \in {\mathcal L}^{1} \cap {\cal L}^2
	\end{equation}
	is a decreasing function that dominates $h_{ij}$ for all $i,j\leq N.$

\begin{rem}
	The existence of $K,\delta$ in \eqref{gassump} excludes instantaneous bursting by imposing a bound on the immediate post-jump intensity. Moreover the existence of $x^*, a^*,c$ in \eqref{eq:doeblin} ensures that no unit will eventually stop spiking. A main example of rate functions that satisfy this assumption are those inducing absolute refractory periods as given in \Cref{Ex:rate}.

	The assumption $ \overline{h} \in {\mathcal L}^1 $ is natural, at
	least in the context of modeling interacting neurons. To obtain stability, it is usually
	assumed that the weight functions are integrable. Here we impose the slightly stronger assumption that $ \bar h_{ij}  \in {\mathcal L}^1 ;$ that is, there exists a decreasing integrable function dominating $ h_{ij } .$
\end{rem}

Throughout this section we use the following notation. For  $K >
0 $ as in \eqref{gassump} above, we denote the PRMs
\begin{equation}\label{eq:pik}
\pi_K ( ds ) := \pi (ds, [0, K]), \quad \pi^i_K ( ds) := \pi^i (ds, [0, K])  \, \mbox{ and } \, \pi_{N K} := \sum_{i=1}^N \pi^i_K .
\end{equation}

\begin{eks}[Hawkes processes with Erlang weight functions]\label{ex:erlang}
Weight functions given by Erlang kernels are widely used in the modeling literature to describe  delay in the information transmission. They are given by
$$ h_{ij } (t) = c_{ij} t^{ n_{ij }} e^{ - \nu_{ij} t },\;\;  t \geq 0, $$
where $ c_{ij} \in \R, \nu_{ij} > 0 $ and $n_{ij} \in \N \cup \{0\} $ are fixed constants. The order of the delay is given by $n_{ij } .$ The delay of
the influence of particle $j$ on particle $i$ is distributed and
taking its maximum absolute value at $n_{ij}/\nu_{ij}$ time units back
in time. The sign of $c_{ij }$ indicates if the influence is
inhibitory or excitatory, and the absolute value of $c_{ij }$ scales
how strong the influence is.
All $h_{ij}$ clearly satisfy \eqref{eq:hbar}.
\end{eks}

The main result of this chapter shows existence of a unique stationary $N-$dimensional  age dependent Hawkes process following the dynamics of \eqref{system}. In order to state the result, we first introduce the notion of \textit{compatibility} (see e.g.\ \cite{bm}). Let
	 $M_{\R_{-}} $ be the set of all bounded measures defined on $ \R_{-} $ equipped with the weak-hat metric and the associated Borel $\sigma-$algebra $ {\cal M}_{\R_-} $ (see Appendix for details). We shall say that $Z$ is \textit{compatible} (to $\pi^{1},\dots,\pi^{N})$ if there is a measurable map $H: M^{N}_{\R_{-}}\con M_{\R_{-}}$ such that for all $t \in \R , $
	\begin{equation}
	\lp \theta^{t} Z\rp_{\vert\R_{-}}=H\lp \lp \theta^{t}\pi^1,\dots,\theta^{t}\pi^{N}\rp_{\vert\R_{-}}\rp
	\end{equation}
	Likewise, we say that a stochastic process $X$ is compatible, if $X_{t}=H\lp \lp \theta^{t}\pi^1,\dots,\theta^{t}\pi^{N}\rp_{\vert\R_{-}}\rp$ for an appropriate measurable mapping  $H$.
\begin{rem}
	Note that if $Z^{1},\dots Z^{n}$ are compatible random measures, then $\lp Z^{1},\dots Z^{n}\rp$ is a stationary and ergodic n-tuple of random measures.
\end{rem}

Let $Z=\lp Z^{i} \rp ,$ $1 \le i\leq N, $ be  compatible random measures on  $\mathbb{R}$. Let $X=\lp X^{i } \rp_{i\leq N }$, $A=\lp A^{i } \rp_{i\leq N }$ be compatible processes defined on  $t\in \R $ such that  $A_{t }^{i }$ is adapted and c\`agl\`ad and  $X_{t }^{i }$ is predictable for  all $1 \le i\leq N$.
We say that  $Z $ is an N-dimensional age dependent Hawkes process on   $t\in \R $, if almost surely
\begin{eqnarray}
	Z^{i} (t_1, t_2 ]  &= &\int_{t_1}^{t_2 }\int_{0}^{\infty}\ql\lb z\leq \psi^{i}\lp X^{i}_{s},A^{i}_{s}\rp\rb  \pi^{i}\lp ds,dz\rp  ,  \nonumber \\
	X^{i}_{t}&=&\sum_{j=1}^{N}\int_{-\infty }^{t-} h_{ij}\lp t-s\rp Z^{j}(ds)   \quad t\in \mathbb{R},\label{systembis}\\
	A^{i}_{t_2} - A^i_{t_1}  &=&  t_2- t_1  -\int_{t_1}^{t_2- }  A^{i}_{t} \, Z^i (dt), \nonumber
	\end{eqnarray}
 for all $-\infty< t_1 \le t_2. $

\begin{s}
\label{Stab}
Grant \Cref{as1}.
\begin{enumerate}
\item  There exists an $N$-dimensional age dependent Hawkes process $ Z $ on $\R,$ compatible to $\lp \pi^{1},\dots , \pi^{N}\rp$.
\item Let $\check{Z}$ be another  $N-$dimensional age dependent Hawkes process with the same weight functions $(h_{ij})_{i, j \le N}$ and driven by the same PRMs $ \lp \pi^{1},\dots , \pi^{N}\rp,$ following the dynamics \eqref{system}, that is,  starting at time $0$ with arbitrary initial ages $(\check A^i_{0})_{i \le N} $ and initial signals  $\lp R^{i} \rp_{ i \le N}  $ such that
\begin{align}\label{eq:rint}
{\mathbb E} \int_{0}^{\infty}\lv R^i_{s}\rv  ds<\infty
\end{align}
for all $ 1 \le i \le N. $ Then almost surely, $\check{Z}$ and $Z$ couple eventually,  i.e.,
\begin{align*}
\exists \, t_{0}\in \R_+ : \;\;  \check{Z}_{| [t_0, \infty) } = Z_{| [t_0, \infty ) }.
\end{align*}
\item If  $Z'$ is another $N$-dimensional age dependent Hawkes process on $\R, $  compatible to $\lp \pi^{1},\dots , \pi^{N}\rp ,$ then  $Z = Z'$ almost surely.
\end{enumerate}
\end{s}
The proof of the above theorem will be given in the next subsection. An immediate corollary of it  is an ergodic theorem for  additive functionals of age dependent Hawkes processes depending only on a finite time horizon.  More precisely, let $ T > 0$ be a fixed time horizon and let $M_T $ be the set of all bounded measures defined on $ ]- T, 0 ], $ equipped with its Borel $\sigma-$algebra $ {\cal M}_T $ (see Appendix).

\begin{koro}
Grant \Cref{as1}. Let $ (Z, X, A) $ be the stationary age dependent Hawkes process and let $\check{Z}$ be as in Item 2.\ of Theorem \ref{Stab}. Let $ f : M_T \to \R$ be any measurable function such that
\begin{equation}\label{eq:intz}
\mu ( f) := {\mathbb E}  f (( Z_{|\,  ]-T, 0 ]} ) ) < \infty .
\end{equation}
Then
\begin{equation}\label{eq:erg}
 \frac1t \int_0^t  f (( \theta^s  \check Z_{|\,  ]-T, 0 ]} ) )ds = \frac1t \int_0^t  f ((   \check Z_{| \, ]s-T, s ]} ) )ds \to \mu ( f)
\end{equation}
almost surely, as $t \to \infty .$
\end{koro}

\begin{proof}
By ergodicity of $ (Z, X, A) $, it holds that
$$ \frac1t \int_0^t  f (( \theta^s   Z_{|\,  ]-T, 0 ]} ) )ds = \frac1t \int_0^t  f ((   Z_{|\,  ] s-T,  s ]} ) )ds \to \mu ( f). $$
Since $\check{Z}_{| [t_0, \infty) } = Z_{| [t_0, \infty)  } ,$ we have that
$$ f (( \theta^s  \check Z_{| \, ]-T, 0 ]} ) ) = f (( \theta^s  Z_{|\,  ]-T, 0 ]} ) ) \mbox{ for all } s  \geq t_0 + T,$$
which implies \eqref{eq:erg}.
\end{proof}

\subsection{Proof of Theorem \ref{Stab}}
This section is devoted to the proof of Theorem \ref{Stab}, and we will thus work under  \Cref{as1}.

The proof of the existence part relies on the Picard iteration
\begin{align}
X^{n,i}_{t}&=\sum_{j=1}^{N}\int_{-\infty}^{t-}h_{ij}\lp t-s\rp Z^{n-1,j} {(ds)} , \nonumber\\
Z^{n,i}\lp t_{1},t_{2}\rf&=\int_{t_{1}}^{t_{2}}\int_{0}^{\infty} \ql\lb z\leq \psi^{i}\lp X^{n,i}_{s},A_{s}^{n,i}\rp\rb \pi^{i}\lp ds,dz\rp ,\quad t_{1}<t_{2}\in \R ,\label{eq:picardheuristics}
\end{align}
where $A^{n,j}$ is the age process of $Z^{n,j}$. The Picard
iteration is similar to the one found in \cite{bm}, but since the
system has an age variable and the intensity is not neccesarily
bounded, the following issues must be addressed before proving convergence.
\begin{itemize}
	\item We need to produce an integrable intensity  $\hat{\lambda}$ that a priori dominates the intensities $\psi^i(X^{n,i}_t, A^{n,i}_t)$. This is done in \Cref{prop:zhat}.
	\item Using the Lipschitz part of \eqref{psias1}, we will construct events $E_t\in \mathcal{F}_t$ for all $t\in \R$ such that $A^{n,i}_t=A^{n+1,i}_t$ on $E_t,$ for all $n\in \N$, $i\leq N$. This is done in \Cref{lem:ok}.
	\item We need to ensure that the $i-$th iteration is well defined, i.e.\ for a given $X^{n,i}$ there exist $Z^{n,i}, A^{n,i}$ such that \eqref{eq:picardheuristics} is satisfied. This is done in \Cref{NoJumpLemma}.
\end{itemize}
Finally we combine these results to complete the proof of \Cref{Stab}. We start with the following useful result which provides bounds on the intensities.
\begin{lm}
	\label{Xbound}
Let $K,\delta$ be the constants from \eqref{gassump}. Let  $X$ be a predictable stochastic process and assume that $(Z,A)$ solves the system
\begin{align*}
Z\lp t_{1},t_{2}\rf =\int_{t_{1}}^{t_{2}}\int_{0}^{\infty} \ql \lb z\leq   \psi \lp X_{s},A_{s}\rp\rb  \pi \lp ds,dz\rp,\quad t_{1}\leq t_{2}\in \R ,
\end{align*}
where $A$ is the age process of $Z$ and where $\psi$ satisfies \eqref{psias1} and \eqref{gassump}.
Suppose moreover that \eqref{eq:hbar} is satisfied. Then almost surely, for any $1 \le i,j\leq N$, $t_{1}\leq t_{2}, $
$$Y_{ij}\lp t_{1},t_{2}\rp=\int_{-\infty}^{t_{1}-} h_{ij}\lp t_{2}-s \rp Z(ds)$$ is well-defined  and
\begin{align}\label{eq:yineq}
\lv Y_{ij}\lp t_{1},t_{2}\rp\rv \leq  \sum^{\infty}_{k=0}\overline{h}_{ij}\lp   t_{2}-t_{1} + A_{t_1} + k \delta\rp+\int_{-\infty}^{t_{1}- A_{t_1} } \overline{h}_{ij}\lp t_{2}-s\rp  \pi_{K}\lp d s\rp.
\end{align}
Moreover,
$$ \E \int_{-\infty}^{t} \overline{h}_{ij}\lp t-s\rp  \pi_{K}\lp d s\rp < \infty $$
for all $t .$
\end{lm}

\begin{koro}\label{cor:hbounded}
If we suppose in addition that $ \bar h (0) < \infty ,$ then
\begin{equation}
 {\mathbb E} Y _{ij}\lp t,t \rp \le K \int_0^\infty \bar h ( u ) du + \sum_{k\geq 0 } \bar h ( k \delta ) < \infty .
\end{equation}

\end{koro}

\begin{proof}[Proof of Lemma \ref{Xbound}]
 For any $i,j\leq N$, $t\leq t_{2}$ we have
\begin{align*}
\lv Y_{ij}\lp t,t_{2}\rp\rv &\leq \int_{-\infty}^{t-} \overline{h}_{ij}\lp t_{2}-s \rp\ql\lb A_{s}\geq \delta\rb Z (ds) +\int_{-\infty}^{t-} \overline{h}_{ij}\lp t_{2}-s \rp\ql\lb A_{s}< \delta\rb Z (ds)\\
&= \int_{-\infty}^{t-A_t} \overline{h}_{ij}\lp t_{2}-s \rp\ql\lb A_{s}\geq \delta\rb Z (ds) +\int_{-\infty}^{t-A_t} \overline{h}_{ij}\lp t_{2}-s \rp\ql\lb A_{s}< \delta\rb Z (ds)\\
&\leq \int_{-\infty}^{t-A_t} \overline{h}_{ij}\lp t_{2}-s \rp G (ds)+\int_{-\infty}^{t-A_t} \overline{h}_{ij}\lp t_{2}-s \rp \pi_{K}\lp ds \rp\\
&:=\hat{Y}_{ij}\lp t,t_{2}\rp+\tilde{Y}_{ij}\lp t,t_{2}\rp ,
\end{align*}
where $G (dt) =\ql \lb A_{t}\geq \delta \rb Z (dt) $. Define now for fixed
$t\in \R $ and for all $ l \in \mathbb{N}, $ $\tau_{l}(t)  := \sup \{
s < \tau_{l-1}(t)  : \Delta G_s = 1 \} , $ where we have put $ \tau_0
(t)  := t-A_t .$ Thus, $ \tau_l (t) $ is the $l$th jump-time of $G$
before $t - A_t  $ -- which is itself the last jump-time of $ Z$ strictly before time $t.$

We may upper bound $\hat{Y}_{ij}$ by
\begin{align*}
\hat{Y}_{ij}\lp t,t_{2}\rp \le  \sum_{l=0}^{G  (- \infty , t ) } \overline{h}_{ij}\lp t_{2}-\tau_{l} \lp t \rp\rp.
\end{align*}
Since $\tau_{l}\lp t \rp-\tau_{l-1}\lp t \rp\geq \delta$ by construction of $G$ and since $\overline{h}_{ij}$ is decreasing, we get the bound
\begin{align*}
\hat{Y}_{ij}\lp t,t_{2}\rp\leq  \sum_{l=0}^{\infty}\overline{h}_{ij}\lp t_2 - t + A_t+ l\delta \rp.
\end{align*}
Note that almost surely, $A_t$ never attains the value 0 for any $t\in
\R$, and in that event, each term in the above sum is finite for all
$t\leq t_2 \in \R$. Moreover, since $\overline{h}_{ij}$ is ${\mathcal L}^{1}$
and decreasing the sum is finite as well.  The expectation  $t\mapsto
{\mathbb E}\tilde{Y}\lp t,t \rp$ is given by
\begin{eqnarray*}
\E \tilde{Y}_{ij}\lp t,t\rp  &=&  \lim_{T\con \infty} {\mathbb E} \int_{-T}^{t- A_t} \overline{h}_{ij}\lp t-s\rp \pi_{K}\lp d s \rp \\
&\le&  \lim_{T\con \infty} {\mathbb E} \int_{-T}^{t} \overline{h}_{ij}\lp t-s\rp \pi_{K}\lp d s \rp=K\int_{0}^{\infty} \overline{h}_{ij}\lp u\rp du<\infty .
\end{eqnarray*}
\end{proof}

We now construct the dominating intensity, as mentioned in the start of the section. Recall that $ L \geq 1  $ is the Lipschitz constant appearing in \eqref{psias1} and let $K, \delta $ be the constants from
\eqref{gassump}, we suppose w.l.o.g. that $  K \geq c,$ where $c$ is the lower bound from \eqref{eq:doeblin}.

\begin{prop}\label{prop:zhat}
Let $ C \geq \max \left\{ 1+ \sum_{k\geq 1 } \bar h ( k \delta ) , K   \right\} .$
There exists a compatible process $( \hat Z , \hat A , \hat \lambda) $ which is defined for any $t \in \R $ by
\begin{equation}\label{eq:hatlambda}
\hat \lambda_t = L \left( C + \int_{-\infty}^{t-} \bar h ( t- s) \pi_{NK} (ds)  + \bar h ( \hat A_t)\right) ,
\end{equation}
where
\begin{equation}\label{eq:hatz}
\hat Z ( t_1, t_2 ] = \sum_{i=1}^N \int_{t_1}^{t_2}  \int_0^\infty  \ql \{ z \le \hat \lambda_s \} \pi^i  \lp ds,dz\rp
\end{equation}
for all $ t_1 \le t_2, $  together with its age process $ \hat A_t.$ Moreover, we have that
\begin{equation}
\E ( \hat \lambda_t ) < \infty .
\end{equation}

\end{prop}

\begin{proof}
By construction, $\hat \lambda_t \geq K $ for all $t$, and therefore,
any jump time $\tau $ of $\pi^i_K $ is also a jump of
  $\hat Z.$ Hence, at $\tau ,$ the age process $ \hat A_t $ is reset to
  $0.$  It is therefore possible to construct a unique solution to
  \eqref{eq:hatlambda} on $t \in \lp \tau,\infty\rp $. This solution
  is non-exploding since the process is stochastically dominated by a
  classical linear Hawkes process $Z' $  having intensity $  L \left( C +
    \int_{-\infty}^{\tau - }  \bar h ( t- s) \pi_{NK} (ds)  + 2  \int_\tau^{t-}  \bar
    h (t - s )   Z' (d s ) \right) $ which is non-exploding by Proposition \ref{prop:13} since $
  \bar h \in {\cal L}^1.$ A  solution on the entire real line  may be
  constructed by pasting together the solutions constructed in between
  the successive jump times of $ \pi_{NK}.$ It is unique and compatible by construction.

It remains to prove that $ \E ( \hat \lambda_t) < \infty .$ Due to stationarity, it is sufficient to prove that $ \E \bar h ( \hat A_0)  < \infty .$ Also from stationarity, writing $ T_1 $ for the first jump time of $\hat Z $ after time $0, $ it follows that $ {\cal L} (T_1 ) = {\cal L} ( \hat  A_0 ) .$
It follows from \Cref{compensator} in the Appendix that
$$ P ( T_1 > t ) = P\lp Z[0,t] = 0 \rp = \E \left( \exp \left(  -
    \int_0^t ( L C + L \xi_s + L \bar h ( \hat A_0 + s ) )ds  \right) \right) ,$$
where $ \xi_s := \int_{ - \infty }^{0-} \overline h  (s-u) \pi_{NK} ( du ) ,$
implying that, since $ \hat A_0 \geq 0 $ and $ \bar h $ is decreasing,
\begin{align*}
\E ( \bar h ( \hat A_0 ) ) &= \int_0^\infty \E \left( \bar h ( t) e^{ - \int_0^t ( L C + L \xi_s + L \bar h ( \hat A_0 + s ) )  ds } (  L C + L \xi_t + L \bar h ( \hat A_0 + t ) )  \right)  dt \\
&\le L \int_0^\infty  \left(  \bar h ( t)^2 +  \bar h (t) \E ( \xi_0) + \bar h ( t) C \right)  dt < \infty ,
\end{align*}
since $ \bar h \in {\cal L}^1 \cap {\cal L}^2  .$
\end{proof}

We now proceed to the construction of events $E_t$ which a priori will serve by coupling the age processes in the Picard iteration. Indeed, Assumption \eqref{eq:doeblin} will enable us to construct common jumps for any two point processes $ Z^1 , Z^2  $ having intensity $\psi ( \tilde X^1_t, A^1_t) $ and $ \psi ( \tilde X_t^2, A_t^2 ),$  where $ A^i_t $ is the age process of $ Z^i  $,  and $\tilde X^i_t $ is a predictable process such that
$$ \psi ( X^i_t, A^i_t ) \le \hat \lambda_t ,$$
for $ i = 1, 2 . $

Fix some $p > a^* $ such that
\begin{align}\label{eq:varrho}
 \sum_{k \geq 1 }   {L} \bar{h } ( p + k \delta ) < \frac{x^*}{3N},
\end{align}
where $a^* $ and $x^* $ are given in \eqref{eq:doeblin}, and fix some $M > LC $ where $L$ and $C$ are as in \eqref{eq:hatlambda}. Then necessarily $M \geq K \geq c  .$
Introduce for all $t \in \R $ the events
\begin{eqnarray}
E^1_t &:= &
\{ \pi^1 ( ds, [0, c ] ) \mbox{ has a unique jump $\tau^1$ in }  ( t - 2N p + p  ,  t- 2N p +  2  p ] )\} \nonumber \\
&& \cap  \bigcap_{j=1}^N \left\{ \int_{ t- 2N p    }^{  \tau^1 - } \int_{\R_+} \ql \{ z \le M \}  \pi^j \lp ds,dz\rp  =   0  \right\}  \nonumber \\
&& \cap \bigcap_{j =1}^N    \left\{ \int_{\tau^1}^{t - 2N p +  2  p} \int_{\R_+} \ql \{ z \le M +  {2L}\bar h (s - \tau^1) \} \pi^j \lp ds,dz\rp = 0 \right\} ,
\end{eqnarray}
and for all $ i = 2, \ldots , N ,$
\begin{eqnarray}\label{eq:ai}
E^i_t &:=& 
\{ \pi^i ( ds, [0, c ]) \mbox{ has a unique jump $\tau^i$ in }  ( t - 2N p + 2 ( i- 1) p + p  ,  t- 2N p +  2 i p ] )\} \nonumber \\
&&\cap  \bigcap_{j=1}^N \left\{ \int_{ t- 2N p +   2 (i-1) p }^{  \tau^i - } \int_{\R_+} \ql \{ z \le M + {2L} \bar h (s- (t- 2N p +   2 (i-1) p) ) \}  \pi^j \lp ds,dz\rp  =   0  \right\}  \nonumber \\
&& \cap \bigcap_{j =1}^N    \left\{ \int_{\tau^i}^{t - 2N p +  2 i p} \int_{\R_+} \ql \{ z \le M +  {2L}\bar h (s - \tau^i) \} \pi^j \lp ds,dz\rp = 0 \right\} ,
\end{eqnarray}
where the constant $c$ is given in \eqref{eq:doeblin}. This event
splits the interval $(t-2Np,t)$ up in intervals of length $2p$, where
the $i$th truncated PRM has exactly one jump in the second part, and
no other events (of truncated PRMs) occur.

To control the past up to time $t-2Np ,$ we also introduce the event
$$ E^0_t := \Big\{ \hat \lambda_{ t- 2N p } +x^{* }  \le {M} \Big\} \cap \Big\{  \int_{- \infty}^{  t- 2N p } \bar h ( t - 2N p - s) \pi_{NK } (ds) \le \frac{x^*}{3N} \Big\} $$
and put
\begin{equation}\label{eq:et}
E_t := \bigcap_{i=0}^N E_t^i .
\end{equation}
The event $E_{2Np}$ is illustrated in Figure \ref{fig:Et}, for
$N=2$ and $ \overline{h}\lp t \rp \approx t^{-0.4}$. The grey area is
the relevant part for the truncated PRMs.

\begin{figure}[H]
	\center
  \includegraphics[scale=0.65]{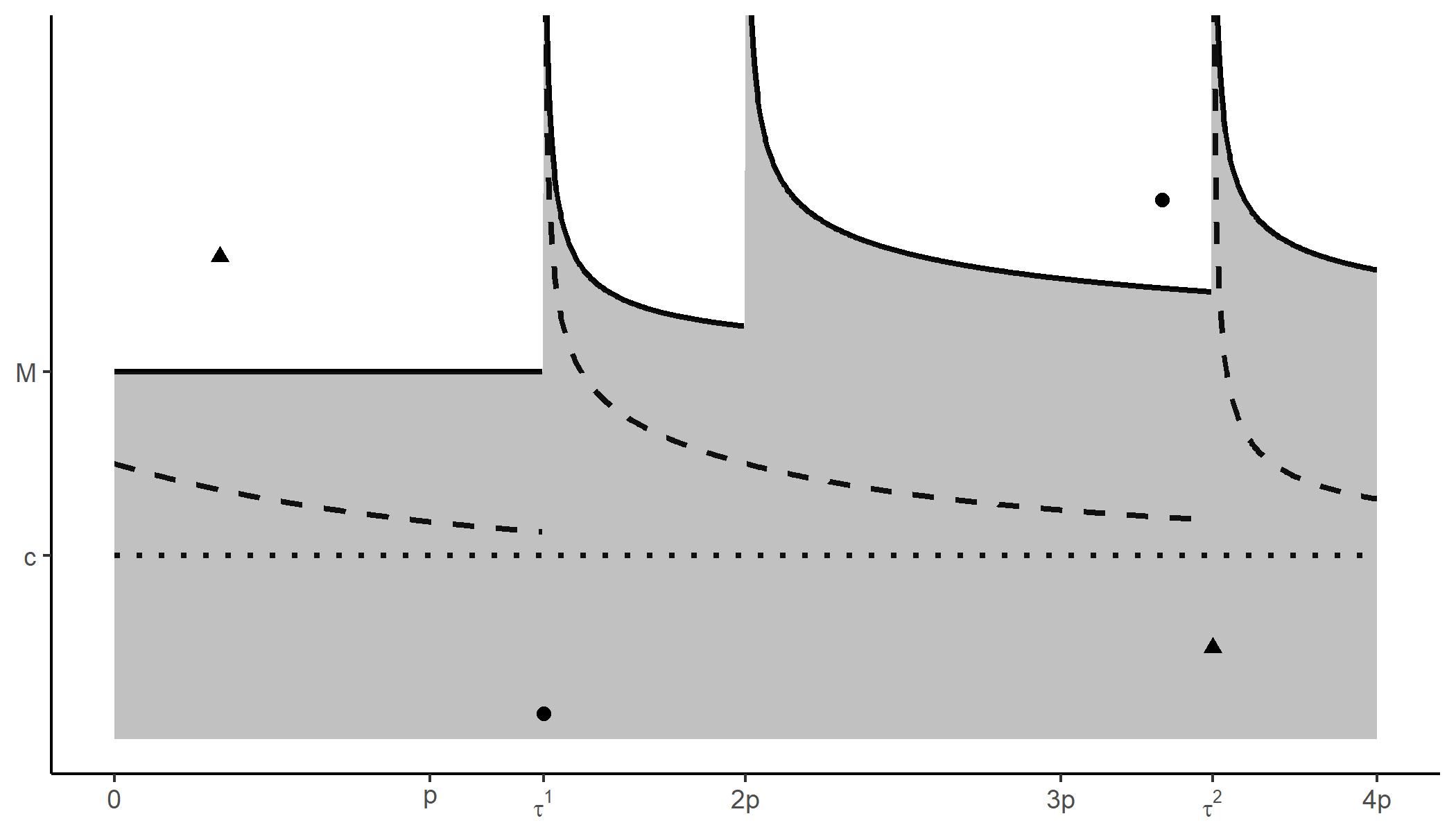}
  \caption[caption]{\label{fig:Et} {\small An illustration of the event $
    E_{2Np }$ with  $N=2$ and $ \overline{h}\lp t \rp \approx
    t^{-0.4}$. The figure shows a superposition of $\pi^{1 }\lp
    \bullet \rp,\pi^{2 }\lp \blacktriangle \rp$, including the jump times
    $\tau^1$ and $\tau^2$, and three curves:
\textbf{The dotted curve }is the constant   $c $. \textbf{The dashed curve} is the intensity process $\hat{\lambda } $. \textbf{The solid curve} is enclosing the area (in grey) of the plane that is relevant to the event $ E_{2Np }$, and it is given by $M  \mbox{ for } 0<s \le \tau^1,
	M + 2L\overline{h} (s-\tau^1) \mbox{ for } \tau^1 < s \le 2p ,\,
	M + 2L\overline{h} (s-2p)\mbox{ for } 2p<s \le \tau^2$ and $
	M + 2L\overline{h} (s-\tau^2) \mbox{ for } \tau^2 < s \le 4p
	$.
}}

\end{figure}

The main feature of the event $ E_{2Np}$ with $N=2$ is the fact that
  the process is forced to have some regeneration events during the
  intervals $ ] p, 2p ] $ and $ ] 3p, 4p ] .$ Indeed, on these
  intervals,
the corresponding age processes will have values larger than $ p >
a^*, $ and the associated memory processes will be bounded by
$ x^* , $ such that we can use \eqref{eq:doeblin}.

Let us return to the general definition of the events $E_t.$  Using induction and the strong Markov property, it follows from integrability of $ \bar
h$ that  $P\lp \bigcap_{i = 0}^{j }E^{i }_{t } \rp>0 $ for all $t\in \R ,j\leq N$. In particular  $P\lp E_{t } \rp>0 $.
Let us define
\begin{equation}\label{eq:yt}
 Y_t :=  \int_{-\infty }^{ t - \hat  A_t  }  \overline{h} (t-s) \pi_{NK }(ds)  + \sum_{ k=0}^{ \infty } \bar h ( \hat  A_t  + k \delta )  .
\end{equation}

We summarize the most important features of the event $E_t$ in the next lemma.

\begin{lm}\label{lem:ok}
On $E_t, $ for all $ 1 \le i \le N, $ each measure $ \pi^i (ds,  [0, c ]) $ has a jump at time $\tau^i \in (t- 2N p, t) $ such that
\begin{equation}\label{eq:1}
 \int_{ t- 2N p}^t \int_{\R_+} \ql \{ s \neq \tau^i , z \le \hat \lambda_s\} \pi^i (ds, dz ) = 0 ,
\end{equation}
\begin{equation}\label{eq:kl}
 \hat \lambda_{\tau^i } \le \hat \lambda_{ t - 2N p } + \frac{2i }{3N} x^*  ,
\end{equation}
\begin{equation}\label{eq:21} Y_{\tau^i } \le \frac{2i }{3N} x^* .  \end{equation}
Moreover,  $ | \tau^i - \tau^{i-1} | \geq a^* , $ where we put $\tau^0 = t - 2N p.$
In particular, take any two point processes $ Z^1 , Z^2$ having intensity $\psi ( \tilde X^1_t, A^1_t) $ and $ \psi ( \tilde X_t^2, A_t^2 ),$  where $ A^i_t $ is the age process of $ Z^i  $,  and $\tilde X^i_t $ is a predictable process such that
$$ \psi ( X^i_t, A^i_t ) \le \hat \lambda_t ,$$
for $ i = 1, 2 . $ It holds that $A^1_t=A^2_t$ under the event $E_t$.
\end{lm}

\begin{proof}
Let  $\lp \tau^{i } \rp_{i \leq N } $ be the jump times as given in the definition of $E_t$. By construction, the inter-distances are at least equal to $p $ and thus strictly larger than $a^*,$ since we chose $ p > a^*$. We shall prove by induction over $j $ that
$$ \int_{ t- 2N p}^{t- 2N p + 2jp} \int_{\R_+} \ql \{ s \neq \tau^i , z \le \hat \lambda_s\} \pi^i (ds, dz ) = 0\;\;\; \forall i \leq N$$
 as well as \eqref{eq:kl} and \eqref{eq:21}  hold for $i \in \{ 0,\dots,j \}  $ in the event $E_{t }$. The induction start is trivial, so assume that the assertion is true up to $j - 1 $. Notice that
 by the induction assumption $$ \hat{\lambda}_{s }  \leq  \hat \lambda_{ t - 2N p } + \frac{2(j-1) }{3N} x^* +2 L\overline{h}\lp s - \tau^{j - 1 } \rp ,  $$
for  $s  \geq \tau^{j - 1 } $ and  until the next jump of $\hat{Z}
$. It follows from the construction of $E^{j }\cap E^{j - 1 } $ that $\hat{Z }\lp \tau^{j- 1 },\tau^{j } \rp = 0 $. This proves the first claim. It also shows that  $\hat A_{\tau^{j } } >p $ so the properties of  $p $ gives
$2L\overline{h }\lp \tau^{j } - \tau^{j - 1 } \rp \leq \frac{2x^{* }
}{3N } $ which implies the remaining claims.

\end{proof}

 The next  result ensures that for a well-behaving process $X^{i }$ there exist couples $\lp Z^{i },A^{i } \rp $
 such that  $Z^{i }$ has intensity  $\psi^{i }\lp X^{i },A^{i } \rp  $
 and  $A^{i }$ is the age of  $ Z^{i}$.  The proof relies on a Picard iteration of \eqref{systembis} that alternately updates  $ \lp X^{i } \rp_{i \leq N } $ and  $ \lp Z^{i },A^{i } \rp_{i \leq N }$.
\begin{lm}\label{NoJumpLemma}
Let $( \hat Z , \hat A , \hat \lambda) $ be as in Proposition \ref{prop:zhat} and let
$\lp X^i_{t}\rp_{t\in \R},  1 \le i \le N , $ be compatible and predictable stochastic processes satisfying that almost surely,
\begin{equation}\label{eq:upper bound}
|X^i_t|  \le  Y_{t } ,
\end{equation}
for all $ 1 \le i \le N, $ $ t \in \R.$

Then there exist random counting measures $Z^i , 1 \le i \le N, $ on $\R$ which are compatible,  and compatible c\`agl\`ad processes $A^i , 1 \le i \le N, $ which almost surely satisfy
\begin{equation}\label{eqtime}
Z^i \lp B \rp = \int_{B}\int_{0}^{\infty} \ql\lb z\leq \psi^i  \lp X^i_{s},A^i_{s}  \rp \rb  \pi^i \lp ds,dz\rp,\quad \forall B\in {\mathcal B}\lp \R\rp ,
\end{equation}
for all $ 1 \le i \le N, $
where $A^i$ is the age process of $Z^i .$
\end{lm}

\begin{proof}  The proof relies on Picard iteration. For that sake, define
  recursively for all $ n \geq  1, $ for all $1 \le i \le N, $

\begin{align}
  Z^{n, i}\lp t_{1},t_{2}\rf=\int_{t_{1}}^{t_{2}}\int_{0}^{\infty} \ql\lb z\leq \psi^i \lp X^i_{s},A_{s}^{n- 1, i }\rp\rb \pi^i \lp ds,dz\rp ,\quad t_{1}<t_{2}\in \R ,
\end{align}
where $A^{n- 1, i}$ is the age process corresponding to $Z^{n- 1 ,
  i}.$ We initialize the iteration with $Z^{0, i}\equiv  \pi^i_K .$

We start by proving inductively over $n$ that the Picard iteration is well-posed, and $Z^{n}$ is non-exploding and compatible.

The induction start is trivial.  We assume that the hypothesis holds
for $n-1. $  Clearly $Z^{n}$ is compatible.   Moreover, $Z^{n, i} $
has intensity $\psi^i ( X^i_t , A^{n- 1, i }_t )  \le L ( 1 +  Y_t), $
and $\mathbb{E} Y_t<\infty$ implying that $ Z^{ni} $
does not explode.

We will now prove the convergence of the above scheme. To do so, define measures $ \underline Z^i $ and $ \overline Z^i $ by
$$ \underline Z^i  [t]  = \liminf_{n\con \infty} Z^{ni}[t]
, \;   \overline Z^i [t]  = \limsup_{n\con \infty} Z^{ni } [t]
$$
for any $ t \in \R ,$ and
$$
\tilde {Z} = \sum_{i=1}^N (    \overline Z^i  -\underline Z^i )  .
$$
That is, $\tilde {Z}$ counts the sum of
the differences of the superior and inferior limit processes. We claim
that $\tilde {Z}$ is almost surely the trivial measure. It will follow
that $Z^{n,i}$, and thus also $A^{n,i}$ converge.

To prove this claim, consider the event
$$
G_{t}=   \left\{  \tilde {Z}\lp t,\infty \rp =0 \right\} .
$$
Notice that $\{ \tilde {Z}\lp t,\infty\rp=0\} = \{ \theta^{t} (\pi^i)_{i=1}^N\in V\}$ for some $V\in {\mathcal M}$ and thus $\{ \tilde {Z}\lp t,\infty\rp = 0 \; \text{infinitely often} \}$ is an invariant set, and thus also a $0 /1$ event.
It follows by standard arguments that $P\lp \tilde {Z}\lp \R \rp =
0\rp = 1$ if $P\lp G_{0}\rp >0$.

We now prove that  $P ( G_0 ) > 0 $ by showing that $ E_0 \subset G_0 , $ where $E_0$ was defined in \eqref{eq:et} above (that is, we choose $t = 0$).

The assumption $ |X_t^i | \le Y_t  $ implies that  $\lambda_t^{n, i } \le \hat \lambda_t $ for all $i, n  $ and $t.$
 Lemma \ref{lem:ok} implies that on  $E_{t } $ we have $ \hat A_{\tau^i } \geq a^* , $ and therefore also $ A^{n, i }_{\tau^i } \geq a^*.$ Moreover, \eqref{eq:21} implies that $ |X^{i }_{\tau^i}| \le x^*.$  Therefore, \eqref{eq:doeblin} implies
\begin{equation}\label{eq:lblambda}
 \lambda_{\tau^i  }^{n, i}  \geq c
\end{equation}
for all $ n, i .$
As a consequence, at time $ \tau^i ,$ all $Z^{ni } $ have a common jump. From \eqref{eq:1} it follows that $ Z^{n i } ( \tau^i , 0) = 0 ,$ and therefore, $ A^{n,i }_{0  } =  - \tau^i . $ In particular, they are all equal. We may now conclude
that on $E_0, $ $ Z^{n,i }_{| \R_+} $ is a constant sequence over $n, $ for all $i.$ In particular, we have $ \tilde Z ( 0, \infty ) = 0.$ To conclude the proof, we have proven that $ E_0 \subset G_0,  $ and thus
$$ P( G_0 \cap  E_0 ) = P( E_0 ) > 0 ,$$
implying the result.
\end{proof}

We are now ready to prove \Cref{Stab}.  \\

{\it Proof of Theorem \ref{Stab}.}

  First we construct a stationary solution to
\eqref{systembis}. For this sake we consider the Picard iteration
\begin{align*}
			X^{n,i}_{t}&=\sum_{j=1}^{N}\int_{-\infty}^{t-}h_{ij}\lp t-s\rp Z^{n-1,j} {(ds)} , \\
			Z^{n,i}\lp t_{1},t_{2}\rf&=\int_{t_{1}}^{t_{2}}\int_{0}^{\infty} \ql\lb z\leq \psi^{i}\lp X^{n,i}_{s},A_{s}^{n,i}\rp\rb \pi^{i}\lp ds,dz\rp ,\quad t_{1}<t_{2}\in \R ,
\end{align*}
where $A^{n,j}$ is the age process of $Z^{n,j}$. We
initialize the iteration with $Z^{0,i} \equiv   \pi^i_K  ,X^{0}\equiv 0.$

We start by proving inductively over $n$ that the Picard iteration is well-posed, $Z^{n,i}$ is non-exploding and compatible, and almost surely
\begin{equation}\label{eq:bound}
L( 1 + |X_t^{n,i }|) \leq \hat{\lambda}_{t}\quad \forall t\in \R,
\end{equation}
for all $n,i$, where $\hat \lambda$ is defined in \eqref{eq:hatlambda} above.
The induction start is trivial. Suppose now that the assertion holds
for $n-1.$  We apply Lemma \ref{NoJumpLemma} with $X^i = X^{n, i },$ and show that the
conditions of this Lemma are met, then well-posedness, ergodicity and
stationarity  of  $Z^{n,i},A^{n,i}$ follow.

Next, we prove the upper bound on $ X^{n, i} .$ By construction,
$$ X_t^{n , i } = \sum_{j=1}^N \int_{-\infty}^{t-} h_{ij} (t-s) Z^{n-1, j } (ds) = \sum_{j=1}^N \int_{-\infty }^{ t - A^{n-1, j }_t } h_{ij} ( t-s) Z^{n-1, j } (ds ) .$$
We apply  Lemma \ref{Xbound} to each of the $N$ terms within the above sum and obtain
\begin{equation}\label{eq:xni}
 \int_{-\infty }^{ t - A^{n-1, j }_t } h_{ij} ( t-s) Z^{n-1, j } (ds ) \le \sum_{k \geq 0} \bar h_{ij} ( A^{n-1, j }_t + k \delta ) + \int_{-\infty}^{t- A_t^{n-1, j } }\bar h_{ij} (t-s) \pi^j_K ( ds ) .
\end{equation}
Since $ \hat \lambda_t \geq \psi^i ( X_t^{n-1, i }, A_t^{n-1, i } ) $ for all $i, $ it follows that $ \hat A_t \le A_t^{n-1, i } $ for all $ i, $
implying that
$$ | X_t^{n, i } | \le  \int_{-\infty}^{t - \hat A_t} \bar h ( t- s ) \pi_{NK }
(ds) +  \sum_{k \geq 0} \bar h ( \hat A_t + k \delta), $$ which is
\eqref{eq:upper bound}.  Finally, since $A^{n-1},Z^{n-1}$ are compatible, it is straight-forward to show that $Z^{n}$ is compatible as well.

Define now
\begin{align*}
\underline{\lambda}^{i}_{t}= \liminf_{n\con \infty} \psi^{i}\lp X^{n,i}_{t},A_{t}^{n,i}\rp,\quad \overline{\lambda}^{i}_{t}= \limsup_{n\con \infty} \psi^{i}\lp X^{n,i}_{t},A_{t}^{n,i}\rp ,  1 \le i \le N .
\end{align*}
Note that by \eqref{psias2} and \eqref{eq:bound}, $ \underline{\lambda}^{i}_{t} \le  \overline{\lambda}^{i}_{t} \le \limsup_{n \to \infty}
L( 1 + |X_t^{n,i }|)
  \leq \hat{\lambda}_{t} .$ So almost
surely, $\underline{\lambda}^{i},\overline{\lambda}^{i}$ have finite
sample paths. Note also that they are limits of predictable processes
(see  \Cref{measurelemma} in Appendix), and thus they are predictable as well. Define also
\begin{align*}
\tilde {Z}^{i} [t] &=  \limsup_{n\con \infty}Z^{n,i}[t] -\liminf_{n \con \infty } Z^{n,i}[t]   = \pi^i \left (\{ t \}\times (  \underline{\lambda}^{i}_{t},\overline{\lambda}_{t}^{i} ]   \right)   ,
\end{align*}
for  $i\leq N,t \in \R$. That is, $\tilde {Z}^{i}$ counts the difference of the superior and inferior limit process. We claim that \begin{equation}\label{eq:tildeZ}
\tilde{Z}=\sum_{j=1}^{N} \tilde {Z}^{j}
\end{equation}
is almost surely the zero-measure. It will follow that $Z^{n,j}$,
and thus also $X^{n,i},A^{n,i}$ converge. Moreover, it is straight
forward to check that the limit variables solve \eqref{systembis}.

To prove this claim, note that
we may also find measurable $H^{i}: M_{\R\times\R_{+}}\rightarrow \R^{2} $ such that almost surely
\begin{align*}
H^{i}\lp \theta^{ t}(\pi^i)_{i=1}^N \rp = \lp \underline{\lambda}^{i}_{t},\overline{\lambda}^{i}_{t}\rp,\quad  \forall t\in \R.
\end{align*}
Consider the events $E_t$ defined in \eqref{eq:et} above as well as
$$
G_{t}= \lp \tilde {Z}\lp t,\infty \rp =0 \rp.
$$
Using the functionals obtained previously, it follows that $\{ \tilde {Z}\lp t,\infty\rp=0\} = \{ \theta^{t}(\pi^i)_{i=1}^N \in V\}$ for some $V\in {\mathcal M}$ and thus $\{ \tilde {Z}\lp t,\infty\rp = 0 \; \text{infinitely often} \}$ is an invariant set, and thus also a $0 /1$ event. As before this implies that $P\lp \tilde {Z}\lp \R \rp = 0\rp = 1$ if $P\lp G_{0} \cap E_0 \rp >0$.

To prove that $P ( G_0 \cap E_0  ) > 0, $ note that we have $\lambda^{n, i }_t = \psi ( X_t^{n, i }, A_t^{n , i } ) \le \hat \lambda_t $ and $ |X_t^{n, i } |\le Y_t .$ Therefore, the same arguments as those exposed in the proof
of Lemma \ref{NoJumpLemma}, show that on  $E_{0}$,  we have $A^{n ,i }_{0} = A^{m , i }_0  ;  $ for all $ n , m $ and $ i, $ that is, the age variables are all equal at time $0.$

Moreover,  on $ G_0,$ either no jumps happen any more, or they happen conjointly, and so the Lipschitz criterion \eqref{psias1} ensures the bound
\begin{align*}
\lv \overline{\lambda}^{i}_{t}-\underline{\lambda}^{i}_{t}\rv\leq L \lim_{n\con \infty}\sup_{m,k\geq n}\lv X^{m}_{t}-X^{k}_{t}\rv \leq L\int_{-\infty }^{0-} h\lp t-s\rp  \tilde {Z}(ds):=\bar{X}_{t},
\end{align*}
for all $ t \geq 0,$ which holds on $G_0 \cap E_0.$
Therefore we may write
\begin{align*}
P\lp G_{0} \cap E_0 \rp \geq  P\lp  E _{0} \cap  \left\{  \sum_{j=1}^{N}\int_{0}^{\infty}\int_{0}^{\infty}\ql\lb z\in \lp \underline{\lambda}^{j}_{s},\underline{\lambda}^{j}_{s}+\bar{X}_{s}\rf \rb  \pi^{j}\lp ds,dz\rp=0\right\} \rp .
\end{align*}
Note that \Cref{compensator} in Appendix reveals that the compensator of the integral-sum above is
\begin{align*}
t\mapsto N\int_{0}^{t} \bar{X}_{s} ds=NL \int_{0}^{t} \int_{-\infty}^{0-} h\lp s-u\rp   \tilde {Z}(du) ds.
\end{align*}
The same lemma  gives an expression for $P ( \bar Z ( 0, \infty) = 0 | \mathcal F_0 ) , $  and so implies the lower bound
\begin{align*}
P\lp G_{0}\cap E_0 \rp&\geq {\mathbb E} \ql \lb E_{0}  \rb \exp\lp -NL \int_{0}^{\infty}\int_{-\infty}^{0-}h \lp s-u\rp \tilde {Z}(du)ds \rp .
\end{align*}

To prove that the right hand side is positive, it suffices to show that the double integral inside the exponential is almost surely finite.
Notice that by construction, $ \tilde Z_t \le \hat Z_t,$ and recall from \Cref{prop:zhat} that  $ \hat\lambda$ is stationary with  $\mathbb{E}\hat\lambda_{0 }<\infty$. After taking expectation,
\begin{eqnarray}
{\mathbb E}\int_{0}^{\infty}\int_{-\infty}^{0-} h\lp s-u\rp \tilde {Z}(du)ds & \le&\label{lastintegrals}
{\mathbb E}\int_{0}^{\infty}\int_{-\infty}^{0-} h\lp s-u\rp \hat  {Z}(du)ds  \nonumber\\
&=& {\mathbb E}\int_{0}^{\infty}\int_{-\infty}^{0-} h\lp s-u\rp     \hat \lambda_u  du ds\nonumber \\
&=&  \E ( \hat \lambda_0 ) \int_{0}^{\infty}\int_{-\infty}^{0-} h\lp s-u\rp   du ds\nonumber\\
& =&\E ( \hat \lambda_0 )  \int_0^\infty t  h (t ) dt < \infty \nonumber.
\end{eqnarray}
This proves the desired result.

We now prove the coupling part. This will be done in two steps.
First suppose that $ |R_t | $ is  bounded by a constant
$C_R.$ We suppose w.l.o.g.\ that $\hat \lambda $ defined in
\eqref{eq:hatlambda} is such that also $ C \geq C_R.$

Let $\lp \lp \check Z^{i}\rp_{i\leq N},\lp \check X^{ i}\rp,\lp \check A^{ i}\rp_{i\leq N}\rp$ be the $N$-dimensional age dependent Hawkes process with initial conditions $\lp \check A^{i}_{0}\rp,\lp R^{i}\rp$ and denote
$\check \lambda^i_t := \psi ( \check X^i_t , \check A_t^i )  . $ Then we clearly have that
$$ \check \lambda^i _0 \le \hat \lambda_0  $$
for all $ i ,$ and it can be shown inductively over the successive jumps of $ \hat Z, $ using Lemma \ref{Xbound}, that this inequality is preserved over time, that is,
\begin{equation}\label{eq:good}
 \check \lambda^i_t \le \hat \lambda_t
\end{equation}
for all $ t \geq 0.$ Now introduce
$$
\check E^0_t := \left\{\sum_{i=1}^N \int_{t - 2N p }^t \ql \left \{ z
    \le \frac{3 N  c}{ x^*  } | R _s| \right \} \pi^i \lp ds,dz\rp = 0  \right\} $$
and put
\begin{equation}\label{eq:etprime}
E'_t := E_t \cap \check E^0_t .
\end{equation}
Let $\lp   \tau^i \rp_{i \le N} $ be the  jump times from \Cref{lem:ok}. We necessarily have that on $E'_t, $
$$ | R_{ \tau^i }| \le \frac{x^*}{  3 N}  ,$$
for all $ 1 \le i \le N.$ As a consequence,  \eqref{eq:21} implies that
$$ | \check X_{\tau^i }| \le Y_{\tau^i } + |R_{\tau^i }| \le x^* .$$
Due to \eqref{eq:good},  we have $ \check A^i_{\tau^i } \geq \hat
A_{\tau^i } \geq a^* . $ Therefore, using \eqref{eq:doeblin}, we conclude that
$$ \check \lambda^i_{\tau^i } \geq c ,$$
implying that $ \tau^i $ is also a jump of $ \check Z^i.$ Thus, on $ E_t', $ at time $ t, $ all $ ( \check A^i , A^i ) , 1 \le i \le N, $  are coupled. Therefore, we have a Lipschitz bound under the event $E'_{t}; $ so with  $Z=\sum_{j=1}^{N}Z^{j},$  $\check {Z}=\sum_{j=1}^{N}\check{Z}^{j},$ we may write
\begin{multline*}
\sum_{i=1}^N \lv  \psi^{i}\lp X^{i}_{t},A^{i}_{t}\rp-\psi^{i}\lp \check X^{i}_{t}, \check A^{i}_{t}\rp\rv
\leq L \sum_{i=1}^N \lv X^{i}_{t}-\check  X ^{i}_{t} \rv
\\
\leq L\lp  \int_{-\infty}^{0-} h\lp t-s\rp  Z(ds) + \lv R_{t}\rv +  \int_{0}^{t-}h\lp t-s\rp   \tilde Z(ds)\rp ,
\end{multline*}
which holds on $E_t' . $

As before $ \tilde Z := | Z - \check Z | $ and
$ G_t := \{ \tilde Z (t, \infty ) = 0 \}  .$
Equivalent considerations as in the first part of the proof yield
\begin{align}
P&\lp G_{t} \cap E'_t \; \vert \; {\mathcal F}_{t}\rp  \nonumber \\
&\geq\ql \lb E'_{t}\rb  \exp\lp -L\int_{t}^{\infty} \left[  \int_{-\infty}^{0} h\lp s-u\rp Z(du) + \lv R_{s}\rv + \int_{0}^{t-} h\lp s-u \rp \tilde{Z}(du)\right]  ds \rp  \nonumber \\
&\geq  \ql \lb E'_{t}\rb  \exp\lp -L\int_{t}^{\infty} \left[  \int_{-\infty}^{t- } h\lp s-u\rp Z(du) + \lv R_{s}\rv + \int_{0}^{t-} h\lp s-u \rp \check {Z}(du)\right]  ds \rp .
\end{align}
Since $ \check \lambda_t \le \hat \lambda_t $ for all $ t \geq 0, $
$$  \int_{0 }^{t- } h\lp s-u\rp  \check Z(du) \le
\int_{-\infty}^{t- } h\lp s-u\rp \hat Z(du)   . $$
As a consequence,
\begin{multline*}
\int_{t}^{\infty} \left[  \int_{-\infty}^{t- } h\lp s-u\rp Z(du) + \lv R_{s}\rv + \int_{0}^{t-} h\lp s-u \rp  \check {Z}(du)\right]  ds \\
\le  2 \int_{t}^{\infty} \left[  \int_{-\infty}^{t- } h\lp s-u\rp  \hat Z(du)  \right] ds + \int_t^\infty  |R_s|    ds
=  C_t +D_t ,
\end{multline*}
where
$$ C_t :=  2 \int_{t}^{\infty} \left[  \int_{-\infty}^{t- } h\lp s-u\rp  \hat Z(du)  \right] ds $$
is stationary and ergodic, and where
$$ D_t := \int_t^\infty  |R_s|   ds.$$
Clearly,
$ D_t \to 0 $
as $t \to \infty $ almost surely.

Now apply \Cref{ergodiclemma} in Appendix with $ U_t := \ql {E_t} e^{-L C_t}, $ $ r_t :=  \ql {E_t} e^{-LC_t} - \ql {E_t'} e^{- L C_t - L D_t }.$ Clearly, $U_t $ is ergodic and satisfies $ P ( U_t > 0 ) > 0.$ To see that $r_t \to 0 $ almost surely as $t \to \infty , $ it suffices to prove that $ \ql E'_t - \ql E_t \to 0 $ almost surely, as $t \to \infty , $ which is equivalent to proving that  $  \ql \check E^0_t \to 1 $ almost surely.  But this follows from
$$ \sum_{i=1}^N \E \int_0^\infty  \int_{\R_+} \ql \left \{ z \le \frac{3N
}{x^* } |R _s| \right \} \pi^i \lp ds,dz\rp < \infty ,$$
which follows from \eqref{eq:rint}.

This finishes the first part of the proof of the coupling result.
Finally, suppose that the initial process only satisfies \eqref{eq:rint}. Take then a sequence $ (\check Z^m , \check X^m , \check A^m ) $ of $ N-$dimensional age dependent Hawkes processes with starting condition $ ( \check A^i_0)_{ i \le n } $ and initial processes $ (-m\vee R^i \wedge m)_{i \le N } .$ Write $ \check \lambda_t^{ m , i} = \psi^i ( \check X^{m, i }_t, \check A_t^{m, i } ) $ for the associated intensity and  $ \check \lambda_t^{m \wedge m+1, i } := \check \lambda^{m, i}_t \wedge \check \lambda^{m+1, i }_t .$ Denote $R_t^{m,i} := -m \vee R_t^i \wedge m.$  As  in the proof of Proposition \ref{prop:13}, we have that
$$
P ( \check Z^m \neq \check Z^{m+1} )
\le L \sum_{i=1}^N   \E \int_0^\infty | R^{m, i}_t  - R^{m+1 , i}_t | dt  .
$$
Since $ \E \int_0^\infty |R^i_t | dt < \infty , $ we conclude that
$$ \sum_m P ( \check Z^m \neq \check Z^{m+1} ) < \infty , $$
implying that almost surely, $ \check Z^m = \check Z $ on $ \R_+ $ for sufficiently large $m. $ Since $ \check Z^m $ and $ Z$ couple eventually almost surely, this proves the coupling part. It remains to prove uniqueness of the stationary solution. Let  $\lp Z',X',A' \rp $ be another age dependent Hawkes process on $t\in \R$ compatible to $\pi^{1},\dots , \pi^{N}$. \Cref{Xbound} gives the inequality
\begin{equation*}
\lv X^{'i }_{t }\rv \leq \sum_{j=1 }^{N }\lp \sum_{k \geq 0} \bar h_{ij} ( A^{'j }_t + k \delta ) + \int_{-\infty}^{t- A^{'j }_t }\bar h_{ij} (t-s) \pi^j_K ( ds ) \rp.
\end{equation*}
Let $\tau $ be a jump of  $\hat{Z}$. Using the above inequality, it is shown inductively over future jumps of $\hat{Z}$ that  $\hat{\lambda}_{t }\geq \lv \lambda^{'i }\rv$ for  all $t\in  ]  \tau,\infty [ $. Thus it follows that  $\hat{\lambda}_{t }\geq \lv \lambda^{'i }\rv$ for all  $t\in \R$.
Note that the  $ \lp Z',X',A' \rp $ system may be written in terms of \Cref{defiHawkes} with initial signals  $R^{'i }_{t }:= \sum_{j=1 }^{N }\int_{\infty }^{0 }h_{ij }\lp t-s \rp dZ^{'j }_{s }$. The same arguments as in \eqref{lastintegrals} give
\begin{equation*}
\mathbb{E}\int_{0 }^{\infty }\lv R^{'i}_{s }\rv ds \leq \mathbb{E}\int_{0 }^{\infty }\int_{-\infty }^{0- } h\lp s-u \rp \hat{Z}\lp du \rp ds<\infty.
\end{equation*}

Therefore it follows from the 2nd point of this theorem that
\begin{equation}
P\lp \exists \, t_{0}\in \R : \;\;  Z'_{| [t_0, \infty) } = Z_{| [t_0, \infty ) }\rp= P\lp \bigcup_{n=-\infty}^{\infty}\lp Z'_{| [n, \infty) } = Z_{| [n, \infty ) }\rp\rp  = 1.\label{uniqueeq}
\end{equation}

Since $Z $ and  $Z' $ are both compatible, it follows that $\lp Z,Z'\rp$ is compatible and therefore also stationary. Thus, the events $\lb Z'_{| [n, \infty) } = Z_{| [n, \infty ) }\rb$ have the same probability for all $n\in \mathbb{Z}$, and from \eqref{uniqueeq} it follows that the probability is equal to $1$. This proves that $Z=Z'$ almost surely.
$\qed$

\subsection{Age dependent Hawkes processes with Erlang weight functions}\label{sec:Erlang}
 Here we show how Theorem \ref{Stab} can be applied for weight functions
given by Erlang kernels as in Example \ref{ex:erlang}, and consider a
one-dimensional ($N=1$) age dependent Hawkes process $(Z, X, A), $
solution of
	\begin{eqnarray}
	Z_t  &= &\int_{0}^{t}\int_{0}^{\infty}\ql\lb z\leq \psi \lp
                  X_{s},A_{s}\rp\rb  \pi \lp ds,dz\rp, \nonumber \\
X_{t}&=&\int_{0}^{t-} h_{}\lp t-s\rp Z(ds) +R_t, \label{systemerlang}\\
	A_{t}&=&A_0 +  t-\int_{0}^{t-} A_{s}  Z(ds) , \nonumber
	\end{eqnarray}
where
\begin{equation}\label{eq:erlangkernel2}
h(t) = b \frac{t^n }{n!} e^{- \nu t } ,
\end{equation}
for some $b \in \R  , \nu > 0 $ and $ n \geq 0 ,$ and where the initial signal is given by
$$ R_t = \int_{- \infty}^0 h(t-s) z (ds)  ,$$
for some fixed discrete point measure $z$ defined on $  ( - \infty , 0 ) $ such that $ \int_{- \infty}^0 h(t- s) z (ds) $ is well defined.

The process $ ( X_{t}, A_{t} ) $ is not Markov, but it is well-known (see e.g.\ \cite{SusEva}) that it  can be completed to a Markovian system $ (
X_t^{(0)} := X_t, X_t^{(1) }, \ldots , X_t^{(n)}, A_t) ,$ by
introducing the auxiliary processes
$$ X_t^{(k) } := \int_{0}^{t-} b \frac{ (t-s)^{n-k}}{(n-k)!} e^{ - \nu ( t-s) }  Z(ds) +  \int_{- \infty}^0b \frac{ (t-s)^{n-k}}{(n-k)!} e^{ - \nu ( t-s) }  z (ds) , \mbox{ for all }  0 \le k \le n.$$

By \cite{SusEva},
these satisfy the system of coupled differential equations, driven by
the PRM $\pi, $ given by
\begin{eqnarray}
d X_{t+}^{(k)} &=& - \nu X_{t+}^{(k)}dt + X_{t+}^{(k+1)} dt , \quad 0 \le k <
n , \\
d X_{t+}^{(n)} &=& - \nu X_{t+}^{(n)}dt + b   \int_{0}^{\infty}\ql\{ z\leq \psi (
                  X^{(0)}_{t},A_{t} ) \}  \pi \lp dt,dz\rp,
\end{eqnarray}
and
$$ A_{t}-A_{0} = t-\int_{0}^{t-} \int_{0}^{\infty} A_s \ql\lb z\leq \psi \lp
                  X_{s},A_{s}\rp\rb  \pi \lp ds,dz\rp =   t-\int_{0}^{t-} A_{s}  Z(ds),$$
for $ t \geq 0. $ Evidently, $h$ satisfies \eqref{eq:hbar}. We suppose
that $\psi  ( x, a) $ satisfies \eqref{psias1} and we strengthen
\eqref{eq:doeblin} to the following assumption.

\begin{ass}\label{eq:psi}
$\psi (x, a ) $ is continuous in $x$ and $a; $ and $\psi ( x, a ) \geq c > 0 $ for all $ x, a$ with  $ a \geq a^* .$
\end{ass}

Then the following result strengthens Theorem \ref{Stab} in this Markovian setting.

\begin{s}
Grant Assumptions \ref{as1} and  \ref{eq:psi}. Then the process $ (
X_t^{(0)} , X_t^{(1) }, \ldots , X_t^{(n)}, A_t)  $ is positively recurrent in the sense of Harris having unique invariant probability measure $\mu .$
\end{s}

\begin{proof}
{\it Step 1.}
By Lemma \ref{Xbound} and Corollary \ref{cor:hbounded}, $ t \mapsto
{\mathbb E}  ( \lambda_t)   = \E ( \psi ( X_t, A_t) )$ and $ t \mapsto
\E ( |X_t|) = \E(|X_t^{(0)}|)  $ are bounded on $ \R.$ By the same
argument, also $ t \mapsto  \E(|X_t^{(k)}|)  $ is bounded for $ 1
\le k \le n.$ Therefore, $  ( X_t^{(0)} , X_t^{(1) }, \ldots ,
X_t^{(n)}) $ is a $1-$ultimately bounded Feller process (the Feller
property follows from the continuity of $ \psi $), see e.g.\
\cite{miha}.

We write $ x = ( x^0, \ldots , x^n ) \in
\R^{n+1} $ for the elements of $ \R^{n+1} $ and denote by $ P_t ( (x, a ) , \cdot ) $ the transition
semigroup of $ (X^{(0)}_t, \ldots , X_t^{(n)} ,  A_t ) .$ Let $ B_k
= \{ (x, a ) : | x|  + |a| \le k \} .$ Then for any $ x_0 \in \R^{n+1}, a_0 \geq 0, $
$$ P_t \Big( (x_0, a_0) , B_k^c \Big) \le P_{(x_0, a_0)} \left( |X_t| \geq \frac{k}{2} \right) + P_{(x_0, a_0)} \left(  A_t \geq \frac{k}{2} \right) ,$$
where
$$ P_{(x_0, a_0)} \left( |X_t| \geq \frac{k}{2} \right) \le \frac{ 2 \sup_t \E_{(x_0, a_0 ) }  ( |X_t | )}{k} ,$$
and
$$ P_{(x_0, a_0)} \left(  A_t \geq \frac{k}{2} \right) \le e^{ - c ( \frac{k}{2}-a^*)_+} ,$$
implying inequality (6) of \cite{miha}. Thus, by  Theorem 1 of
\cite{miha},  $ (X^{(0)}_t, \ldots , X_t^{(n)} ,  A_t)_{t \geq 0}  $
possesses invariant probability measures $\mu$ (not necessarily unique ones).\footnote{As a matter of fact, this provides
a different approach to prove the existence of a stationary version of the age dependent Hawkes process.}

{\it Step 2.} We shall now use the coupling property proved in Theorem \ref{Stab} to prove uniqueness of the invariant measure $\mu .$ In what follows we write $ (Z, X, A ) $ for the stationary version of \eqref{systemerlang}, which exists according to Theorem \ref{Stab}. Moreover, we write $ (\tilde Z, \tilde X, \tilde A ) $ for a version of \eqref{systemerlang} starting at time $t= 0 $ from an arbitrary initial age $ a_0$ and an initial configuration $x_0= (x_0^{(0)}, \ldots , x_0^{(n)} ) $ with $ x_0^{(k)} = b\int_{- \infty }^0 \frac{ (-s)^{n-k}}{(n-k)!} e^{-\nu s } z (ds ) .$  Write
$$ \tau_c := \inf \{ t > 0 : Z \mbox{ and } \tilde Z \mbox{ couple at time } t \}\vee 1 .$$
Note that
$$| h(s+u )| \le C |h (s)| | h(u)| $$
for all $  s,u \geq 1, $ where $C$ is an appropriate constant. It follows that almost surely, for all  $ t \geq \tau_{c } + 1$
$$ |X_t - \tilde X_t| \le   \overline{h} ( t- \tau_c)  (Z ( [0, \tau_c]) +
\tilde Z ( [0, \tau_c] ))\tilde Z ( [0, \tau_c] )) + C |h(t- \tau_{c})| ( |X_{\tau_{c }} | + | \tilde X_{\tau_{c }} |) , $$
showing that
\begin{equation}\label{eq:close}
\lim_{t \to \infty }  |X_t - \tilde X_t|  = 0
\end{equation}
almost surely, since $| h ( t - \tau_{c } )| \to 0 $ as $ t \to \infty.$  In the same way one proves that also
\begin{equation}\label{eq:closebis}
\lim_{t \to \infty }  |X^{(k)}_t - \tilde X^{(k)}_t|  = 0
\end{equation}
almost surely, for all $ 1 \le k \le n.$  Moreover, we obviously have that $ \tilde A = A $ on $ [ T_1 \circ \theta_{\tau_c}, \infty [  , $ where $  T_1 \circ \theta_{\tau_c} = \inf \{ t > \tau_c : Z ( [t]) = \tilde Z ( [t] ) = 1 \}.$ Since $ \psi( x, a ) \geq c > 0 $ for all $ a \geq a^*, $ $ T_1 \circ \theta_{\tau_c} < \infty $ almost surely. This implies the uniqueness of the invariant measure.

{\it Step 3.} Finally, to prove the Harris recurrence of the process $ ( X^{(0)}_t,\ldots , X_t^{(n)} ,  A_t) ,$ we rely on the following local Doeblin lower bound. It states that for all $ (x^{**}, a^{**} )   \in \R^{n+1}\times \R_+ , $ there exist $R > 0 , $ an open set $ I \subset \R^{n+1}\times \R_+ $ and a constant $\beta \in ] 0, 1 [, $ such that for any $T > (n+2) a^* , $
\begin{equation}\label{doblinminorization}
P_{T} ( (x_0, a_0) , \cdot ) \geq \beta \ql_C  ( x_0, a_0) U ( \cdot) ,
\end{equation}
where $ C = B_R  (x^{**}, a^{**} ) $ is the (open) ball of radius $R$ centered at $(x^{**}, a^{**} )  ,$ and where $ U $ is the uniform measure on $  I.$ This lower bound follows easily adapting the proof of Theorem 3 in \cite{aline} to our framework.

We may apply the above result with $ (x^{**}, a^{**} ) \in supp (\mu ) $ where $\mu$ is the (unique) invariant measure of the process. Then for the stationary version of the process, $(X_t^{(0)}, \ldots, X_t^{(n)} ,  A_t)  \in B_{R/2}  (x^{**}, a^{**} ) $ infinitely often. Then \eqref{eq:close} and \eqref{eq:closebis} imply that also $(\tilde X^{(0)}_t, \ldots , \tilde X_t^{(n)} ,  \tilde A_t)  \in B_{R} ( x^{**}, a^{**} )  = C $ infinitely often, almost surely. The classical regeneration technique, see e.g.\ \cite{nummelin},  allows to conclude that indeed the process is positively recurrent in the sense of Harris.
\end{proof}

\section{Mean-field limit and propagation of chaos}\label{sec:mf}
In this section we focus on a multi-class mean-field setup of the age dependent Hawkes process on $\R_+,$ starting from initial signals and ages. We propose a limit system, and show how the
high dimensional system couples with the limit system. This can be
seen as a generalization of the work of Chevallier \cite{chevallier}
where a single class is considered under the assumption that the
spiking rate function $\psi$ is uniformly bounded. The multi-class
setup is similar to the one in \cite{SusEva} for ordinary Hawkes
Processes, and to the data transmission model in \cite{carl2}. We
also discuss how to approximate a Hawkes process induced by one weight
function, by another Hawkes process, induced by different weight
functions.

\textbf{Setup in this section:}
	In addition to the fundamental
	assumptions, we introduce the following specifications for the
	mean-field setup, which will be used throughout this section.  We
	partition the indices of individual units into $\ccK \in \N$ different
	populations, where $\ccK > 0$. More precisely, for each fixed total population size $N \in \N ,$
	$$
	N_{k}:=N_{k}\lp N\rp:=\#\lb i\leq N : i\text{ in population } k \rb
	$$
	will denote the number of units belonging to population $k,  1 \le k \le \ccK,$ and
	$$ N = N_1 + \ldots + N_{\ccK} .$$
	We assume that each population represents an asymptotic part of all
	units, i.e., there exists $p_{k}>0$ such that
	$$\frac{N_{k}}{N} \conx{N\con \infty} p_{k}.$$
	For a fixed $N\in \N , $ we re-index  the $N$-dimensional age dependent Hawkes process  of \eqref{system} as
	$$\lp Z^{kj}\rp_{k\leq \ccK,j\leq N_{k}}, $$
	where the superscript $ kj $ denotes the $j$th unit in population $k.$  The weight function from the $i$th unit of population $l$ to the $j$th unit of population $k$ is given by $N^{-1}h_{kl}.$ Moreover,  all units within the same population have the same spiking rate  $\psi^{k}.$ Finally, we assume that all units in population $k$ have the same initial signal $R^{k},$ and that the initial ages are interchangeable in groups and mutually independent in and between groups.
	With this set of parameters, the age dependent Hawkes process $\lp Z,X,A\rp$ from \eqref{system} defined on $t\in \R_+$  becomes

	\begin{eqnarray}
	Z_t^{kj} &= &\int_{0}^{t}\int_{0}^{\infty} \ql \lb z\leq \psi^{k}\lp X^{k}_{s},A^{kj}_{s}\rp\rb d\pi^{kj}\lp ds,dz\rp,\quad j\leq N_{k},k\leq \ccK,\label{meanfieldsystem} \\
	X^{k}_{t}&=&\frac{1}{N}\sum_{l=1}^{\ccK}\sum_{j=1}^{N_{l}}\int_{0}^{t-}h_{kl}\lp t-s\rp Z^{lj}(ds)+R_t^{k}, \quad \quad \;\;\; k\leq  \ccK , \nonumber
	\end{eqnarray}
	where $ A^{kj}$ is the age process of $Z^{kj},$ starting from
	$A^{kj}_{0}$ at time $t=0.$ Sometimes, to explicitly indicate the dependency on $N,
	$ we add $N$ to the superscript and write $ Z_t^{Nki } , X^{Nk } $ and $ A^{Nki} .$

\subsubsection*{Model observations}
	 \begin{enumerate}
	 	\item Suppose that the initial ages $(A^i_0 )_{ i \in \N } $ are exchangeable. Then the symmetry of the system gives interchangeability between units within the same population,  i.e.,  $Z^{kj}\stackrel{\mathcal L}{=}Z^{ki}$ for $i,j\leq N_{k},k\leq \ccK$.
	 	\item In the mean-field setup, all units within a population $k$ share the same memory process $X^{k}$.
	 \end{enumerate}

	 \subsection{The Limit System}
	  We propose a limit system for $N\con \infty$.  To pursue this goal,  take finite variation functions $t\mapsto \alpha^{k}_{t}$, locally bounded functions $t\mapsto \beta^{k}_{t}$, and PRMs $\pi^{k}$ for  $k\leq \ccK,$ and consider the stochastic convolution equation
	\begin{eqnarray}
	 \phi^{k}_{t}&=&\int_{0}^{t} {\mathbb E}\psi^{k}\lp x^{k}_{s},A^{k}_{s} \rp ds , \nonumber\\
	 x^{k}_{t}&=&\sum_{l=1}^{\ccK}p_{l}\int_{0}^{t} h_{kl}\lp t-s \rp d\alpha^{l}_{s}+\beta^{k}_{t} , \label{diff}\\
	 A^{k}_{t}-A^{k}_{0}&=&t-\int_{0}^{t-}\int_{0}^{\infty} A^{k}_{s}\ql\lb z\leq \psi^{k}\lp x^{k}_{s},A^{k}_{s}\rp\rb \pi^{k}\lp ds,dz\rp, \nonumber
	\end{eqnarray}
	with unknown $\lp \phi,x,A\rp = \lp \phi^{k}, x^{k}, A^{k}\rp_{k\leq \ccK}$. Notice that only $A$ is stochastic. Introducing
	$$Z_t^k = \int_{0}^{t}\int_{0}^{\infty} \ql\lb z\leq \psi^{k}\lp x^{k}_{s},A^{k}_{s}\rp\rb \pi^{k}\lp ds,dz\rp, $$
	we can interpret $ A^k $ as age process of $Z^k.$ Hence, in the limit, the network activity can be resumed via the deterministic quantities $x^k , 1 \le k \le {\mathcal K} , $ the only remaining randomness is in the individual age processes. Finally, notice that  $\phi$ depends on the law of $A$.

We are motivated by what for the moment is a heuristic.
$$N^{-1}\sum_{j=1}^{N_{k}}Z^{N kj}\approx  p_{k}{\mathbb E} Z^{k}$$
for large $N,$ where $ (Z^1, \ldots, Z^{\ccK} ) $ denotes the limit process such that each $Z^k $ describes the jump activity of a typical unit belonging to population $k.$ This relation invites the idea that the memory process for $N\con \infty$, $t\mapsto x_{t}$ should satisfy the integral system \eqref{diff} with  $\phi^{k}_{t}=\alpha^{k}_{t}={\mathbb E} Z^k _{t}$ and $\beta^{k}_{t} = {\mathbb E} R^{k1}_{t}$. This motivates the following result.

\begin{lm}
	\label{odelm}
	Let $\beta_{t}=\lp \beta^{k}_{t}\rp_{k\leq \ccK}$ be measurable and locally bounded. There is a unique function $\alpha$ such that $\alpha = \phi$, where $\lp \phi,x,A \rp$ is the solution to \eqref{diff}. Moreover, $\phi$ is continuous and $x$ is bounded on $\lf 0,T \rf$ by a constant $C$ which depends  on $h:= \sum_{k, j } | h_{kj } |,$
$\lV \beta\rV_T$, $T$ and $L$.
\end{lm}

The proof is given in the Appendix.  Once this lemma is established, we can ensure existence of the limit process.

\begin{s}\label{koroet}

	Let $\beta_{t}=\lp \beta^{k}_{t}\rp_{k\leq K}$ be measurable and locally bounded. There is a unique solution $\lp Z,A \rp$ to the integral equation
		$$
		Z^{k}_{t} = \int_{0}^{t}\int_{0}^{\infty}\ql\lb z\leq \psi^{k} \lp  \sum_{l=1}^{\ccK}p_{l}\int_{0}^{s} h_{kl}\lp s-u \rp d{\mathbb E} Z^{l}_{u}+\beta^{k}_{s},A^{k}_{s}\rp \rb \pi^{k}\lp ds,dz\rp , 1 \le k \le \ccK,
		$$
		where $A^{k}$ is the age process corresponding to $Z^{k}$, initialized at $A^{k}_{0}$.

\end{s}

	\begin{proof}
		Let $\lp \phi,x,A \rp_{k\leq \ccK}$ be the tuple given in  \Cref{odelm}. Define the counting process
		$$
		Z^{k}_{t}:=\int_{0}^{t}\int_{0}^{\infty}\ql \lb z\leq \psi^{k}\lp x^{k}_{s},A^{k}_{s}\rp\rb \pi^{k}\lp ds,dz\rp.
		$$
		It is clear that $A^{k}$ is the age process of $Z^{k}$, and since $d{\mathbb E} Z^{k}_{t}={\mathbb E}\psi^{k}\lp x^{k}_{t},A^{k}_{t} \rp dt $,  $Z^{k}$ will satisfy the desired identity. For uniqueness, consider another solution $\lp \tilde{Z}^{k},\tilde{A}^{k} \rp_{k\leq \ccK}$, which satisfies the same identity :
		$$
		\tilde{Z}^{k}_{t}= \int_{0}^{t}\int_{0}^{\infty}\ql\lb z\leq \psi^{k} \lp  \sum_{l=1}^{\ccK}p_{l}\int_{0}^{s} h_{kl}\lp s-u \rp d{\mathbb E} \tilde{Z}^{l}_{u}+\beta^{k}_{s},\tilde{A}^{k}_{s}\rp \rb \pi^{k}\lp ds,dz\rp.
		$$
		Defining $\tilde{x}^{k}_{t}=\sum_{l=1}^{\ccK}p_{l}\int_{0}^{t} h_{kl}\lp t-s \rp d{\mathbb E} \tilde{Z}^{l}_{u}+\beta^{k}_{t}$ and $\tilde{\phi}^{k}_{t}=\int_{0}^{t} {\mathbb E} \psi^{k}\lp \tilde{x}^{k}_{s},\tilde{A}^{k}_{s} \rp ds$, we note that ${\mathbb E} \tilde{Z}^{k}_{t}=\tilde{\phi}^{k}_{t}$. Thus, if we insert $\alpha = \tilde{\phi}$ in \eqref{diff}, $\lp \tilde{\phi}, \tilde{x},\tilde{A} \rp$ is a solution, and hence the uniqueness part of \Cref{odelm} gives that $\lp \tilde{\phi}, \tilde{x},\tilde{A} \rp = \lp \phi, x , A \rp$ and thus also $Z = \tilde{Z}$.
	\end{proof}

\subsection{Large network asymptotics and weight approximations}
In this section we couple the $N$-dimensional Hawkes process with the
limit system proposed in the previous section.
This coupling implies that the finite-dimensional system
  converges to the limit system. The result is traditionally named
  {\it Propagation of Chaos}, a typical result within mean-field
  theory.  Specifically for Hawkes processes, there are several
variants of this result. Some of the recent results may be found in
\cite{chevallier}, \cite{dfh} and in \cite{SusEva}.

\subsubsection*{Framework for Propagation of Chaos}

We will first introduce a set of assumptions.

\begin{ass}\label{ass:R}
We  are given, for each $1 \le k\leq \ccK,$ a sequence $\lp
          R^{Nk}\rp_{N\in \N}$ of initial signals with $\sup_{k\leq \ccK,N\in \N} \lV {\mathbb E} R^{Nk} \rV_{t}<\infty,$ such that there is a locally bounded function $t\mapsto r^{k}_{t}$ with
	\begin{align}
	\int_{0}^{t}{\mathbb E} \lv R^{Nk}_{s}-r^{k}_{s}\rv ds \to 0
          \, \mbox{ as } \, N \to \infty, \label{Rconv}
	\end{align}
	for all $t\geq 0$.
\end{ass}

\begin{ass}\label{ass:A}
	The initial ages $ A_0^{k i }, 1 \le k \le \ccK, 1 \le i < \infty $ are i.i.d.
	\end{ass}

	\begin{ass}\label{ass:h}
The weight functions $h^{N}_{kl} : \R_{+}\con \R$ satisfy
$h^{N}_{kl}\conx{} h_{kl}$ as $N \to \infty $ locally in ${\mathcal
  L}^{1},$ where $h_{kl}\in {\mathcal L}^{2}_{loc}$ for all $ 1 \le k,
l \le \ccK$.
\end{ass}

Consider an i.i.d.\ sequence of driving PRMs $ \pi^{kj } , 1 \le k \le \ccK, j \geq 1 .$ Define for each $N\in\N, $ the $N$-dimensional Hawkes process
$$\lp Z^{N},X^{N},A^{N} \rp = \left( Z^{Nki}, X^{Nk}, A^{Nki} \right)_{k \le \ccK, i \le N_k } ,$$
given by \eqref{meanfieldsystem}, driven by $(\pi^{kj }),$ with weight functions $\lp N^{-1 }h^{N}_{kl}\rp$, spiking rate $\lp \psi^{k}\rp$ and initial processes $\lp R^{Nk}\rp.$

Applying \Cref{koroet} with weight functions $\lp h_{kl}\rp$ and initial functions $\beta^{k} = r^{k},$ we obtain, for any $ 1 \le k \le \ccK$ and for all $ i \in \N, $ a solution $(Z^{ki}, X^k, A^{ki} )  $ to the equation
$$ Z^{ki}_{t} = \int_{0}^{t}\int_{0}^{\infty}\ql\lb z\leq \psi^{k} \lp  \sum_{l=1}^{\ccK}p_{l}\int_{0}^{s} h_{kl}\lp s-u \rp d{\mathbb E} Z^{lj }_{u}+r^{k}_{s},A^{ki}_{s}\rp \rb \pi^{ki}\lp ds,dz\rp , $$
$1 \le k \le \ccK, i \in \N ,$ driven by the same sequence of PRMs.

\begin{s}[Propagation of Chaos]\label{theo:prop}
	Consider the framework described above and grant Assumptions \ref{ass:R}--\ref{ass:h}. Then for all $ t \geq 0,$
	\begin{align} \label{conresult}
	 {\mathbb E} \lv d\lp Z^{Nki}_{t}-Z^{ki}_{t} \rp\rv 	\conx{} 0,
	 \text{ for } \; N \con \infty ,
	\end{align}
	for all $k\leq \ccK,i\in \N$. In particular, for any finite set of indices $\lp k_{1}, i_{1},\dots ,k_{n},i_{n}\rp ,$
	we have weak convergence
	\begin{align*}
	\lp Z^{Nk_{1}i_{1}},\dots ,Z^{Nk_{n}i_{n}}\rp_{t\geq 0}\conx{wk}\lp Z^{k_{1}i_{1}},\dots ,Z^{k_{n}i_{n}}\rp_{t\geq 0}
	\end{align*}
	as $ N \to \infty $ (in $ D ( \R_+, \R_+^n ), $ endowed with the topology of locally-uniform convergence).
\end{s}

To prove this theorem we shall need the following lemma.

\begin{lm}	\label{weightapproximation}
 Let $( h_{kl})_{ 1 \le k, l \le \ccK } ,( \tilde{h}_{kl})_{1 \le k,l \le \ccK} $ be sets included in a family ${\mathcal E}$ of real-valued functions defined on $ \R_+$ which is uniformly integrable on $[0,T] .$ Define $( Z,X,A ),( \tilde{Z},\tilde{X},\tilde{A} )$ as the $N-$dimensional age dependent Hawkes process with weight functions $( N^{-1 }h_{kl})_{ 1 \le k, l \le \ccK } ,$ $( N^{-1 }\tilde{h}_{kl})_{1 \le k,l \le \ccK} $, rate functions $\lp \psi^{k}\rp_{ k \le \ccK} $, and with initial conditions $A_{0}$, $ R^{k}$. There exists $C>0$ depending on the family ${\mathcal E} $, on $T,L,\ccK$  and on $\sup_{k\leq \ccK}\lV {\mathbb E} R^{k}\rV_{T}$ (but not on $N$) such that
	\begin{align*}
	\sum_{k=1 }^{\mathcal{K} }{\mathbb E} \lv d\lp Z^{k1}_{t}-\tilde{Z}^{k1}_{t}\rp	\rv\leq C_{T}\sum_{k,l=1 }^{\mathcal{K} }\int_{0}^{t}\lv  h_{kl }-\tilde{h}_{kl }\rv \lp s \rp ds ,
	\end{align*}
	for all $t\leq T. $
\end{lm}

The proofs of \Cref{theo:prop} and \Cref{weightapproximation} may be found in the Appendix.

\begin{rem}
The result shows that finitely many units will be asymptotically independent for $N\con \infty.$
\end{rem}

\subsection{The mean-field limit in the case of a hard refractory period}
In this section we consider the mean-field limit of age dependent
Hawkes processes with one single population ($\ccK= 1 $) and a weight
function given by an Erlang kernel as in Example \ref{ex:erlang}, that is,
$$ h(t)  = b e^{ - \nu t } \frac{t^n }{n!} , $$
for some fixed constants $b \in \R, \nu  > 0 , $ $n \in \N.$
Throughout this section we suppose a hard refractory period of
  length $\delta$ after a jump where no new jumps can occur as given
in the following assumption.
\begin{ass}\label{eq:hard}
$$ \psi (x, a) = f(x) \ql \{ a \geq \delta \}.$$
\end{ass}
We start by rewriting the limit system in this frame. Recall that
\begin{equation}\label{eq:limitlambda}
\phi_t  = \int_0^t E ( \psi ( x_s ,  A_s) ) ds   = \int_0^t \bar \lambda_s ds,
\end{equation}
where
\begin{equation*}
\bar \lambda_t = E ( \psi ( x^{(0)}_t ,  A_t) )
\end{equation*}
denotes the expected number of jumps up to time $t$ of a typical unit in the limit system.
As in Section \ref{sec:Erlang} above, we write $ x^{(0) } := x,  $ and we add auxiliary variables $x^{(i)}, 1 \le i \le n $ to obtain the system
$$  A_{t} = A_0 + t  - \int_0^t \int_{\R_+} A_{s}  \ql {\{ z \le \psi ( x^0_s,   A_{s})\}} \pi\lp ds,dz\rp$$
together with
\begin{eqnarray}\label{eq:cascade}
  dx^{(0)}_t&=&-\nu x^{(0)}_tdt + x^{(1)}_tdt ,\\
 &\vdots& \nonumber \\
dx^{(n-1)}_t&=&-\nu  x^{(n-1)}_tdt + x^{(n)}_tdt , \nonumber\\
dx^{(n)}_t&=&-\nu  x^{(n)}_tdt+b  d \phi_t= -\nu  x^{(n)}_tdt+b   \bar \lambda_t dt .  \nonumber
\end{eqnarray}
Let us now study the age process of this limit system. Write $ \tau_t = \sup \{ 0 \le s \le t : \Delta  A_s \neq 0 \} $ for the last jump time of the process before time $t,$ where by convention, $ \sup \emptyset := 0 .$ Then obviously,
$$ A_{t+} = (t- \tau_t) \ql{\{ \tau_t > 0 \}} + (A_0 + t ) \ql{\{ \tau_t = 0 \}}.$$
Due to Assumption \ref{eq:hard}, we have the following
\begin{prop}
$$ {\mathcal L} ( \tau_t) (dz) = \E( e^{- \int_0^t  f(x_s^{(0)} ) \ql \{
    A_0 + s \geq \delta \} ds } ) \delta_0  (dz) + f(x_z^{(0)}) p_{z}
e^{ - \int_{z+ \delta}^t f( x_s^{(0)}) ds } \ql { \{0 < z < t \}} dz,$$
where $p_t = P ( A_t \geq \delta ) $ is given by
\begin{eqnarray*}
p_t &=& \E \left( \ql {\{ A_0 \geq \delta - t \}} e^{- \int_{(\delta - A_0) \vee 0} ^t  f(x_s^{(0)} ) ds  } \right) + \int_0^{t- \delta} f(x_s^{(0)}) p_s e^{- \int_{s + \delta }^t f( x_u^{(0)}) du } ds \\
&=& \int_{ (\delta - t )\vee 0}^\infty \mu_0 ( da)   e^{- \int_{(\delta - a) \vee 0} ^t  f(x_s^{(0)} ) ds  } + \int_0^{t- \delta} f(x_s^{(0)}) p_s e^{- \int_{s + \delta }^t f( x_u^{(0)}) du } ds ,
\end{eqnarray*}
where $ A_0 \sim \mu_0 ( da)  .$
\end{prop}

In particular, the above representation shows that, even starting from a non-smooth initial trajectory, $p_t$ is eventually smooth.
\begin{koro}
For any starting law $\mu_0 ( da) , $ $t \mapsto p_t$ is continuous on
$ ]  \delta , \infty ],$  and thus, taking into account
\eqref{eq:cascade}, $ C^1 ( ] 2 \delta , \infty [, \R ) ,$
solving
$$ d p_t = - f ( x_t^{(0)} ) p_t dt + f(x_{t- \delta }^{(0)}) p_{t- \delta } dt , \mbox{ for all } t > 2 \delta .$$
\end{koro}

If the starting law is smooth, we can say more.
\begin{koro}
If $ \mu_0 ( da) = \mu_0 (a) da , $ with $\mu_0 \in C ( \R, \R_+), $ then for all $ t < \delta , $
$$ p_t = \int_{ \delta- t}^\infty \mu_0 (a)   e^{- \int_{(\delta - a) \vee 0} ^t  f(x_s^{(0)} ) ds  }  da $$
is continuous and thus, taking into account \eqref{eq:cascade}, $ C^1 ( [0, \delta [, \R ) .$ In particular, on $ [0, \delta [, $ $t \mapsto p_t$ solves
$$ dp_t = \mu_0 (\delta - t ) - f ( x_t^{(0)} ) p_t dt .$$
By induction, this implies that $ t \mapsto p_t $ is continuous on $ \R_+ $ and $C^1 $ on $ ] \delta , \infty [, $ with
$$ d p_t = - f ( x_t^{(0)} ) p_t dt + f(x_{t- \delta }^{(0)}) p_{t- \delta } dt , \mbox{ for all } t > \delta .$$
Moreover, at $t = \delta, $ writing $ \dot p := \frac{d}{dt} p_t, $
$$ \dot p_{\delta - } = \mu_0 ( 0 ) - f ( x_\delta^{(0)} ) p_\delta  \;
\mbox { and }  \; \dot p_{\delta + } = - f ( x_\delta^{(0)} ) p_\delta + f ( x_0^{(0)} ) p_0 .$$
\end{koro}

Let us now look for possible stationary solutions of \eqref{eq:cascade}.  At equilibrium, we necessarily have that
$$ x^{(0)} \equiv x^{**} $$
for a given value $ x^{**} \in \R . $ It follows that  $  A_t $ is a renewal process with dynamics
\begin{equation}\label{eq:renewal}
 d  A_{t+} = dt -  A_{t} \int_{\R_+} 1_{\{ z \le \psi ( x^{** },  A_{t})\}}  \pi \lp dt,dz\rp .
\end{equation}
$A$ is recurrent in the sense of Harris if it comes back to $0$ infinitely often almost surely. This happens if $ \int_0^\infty \psi ( x^{**} ,  A_{t})  dt = \infty $ almost surely, which is granted by the following condition.

\begin{ass}
For all $x, $ there exists $ r(x) \geq 0 $ such that $ \psi ( x, a) $ is lower bounded for all $ a \geq r( x) .$
\end{ass}

The stationary distribution of  \eqref{eq:renewal} is absolutely continuous with respect to the Lebesgue measure on $ \R_+,$ having the density (see Proposition 21 of \cite{evafou})
$$ g_{x^{** }} (a) = \kappa  e^{ - \int_0^a \psi ( x^{**} , z) dz } $$
on $\R_+,$  where $ \kappa $ is chosen such that $ \int_0^\infty g_{x^{**} } (a) da = 1.$  Recall that
$ \bar \lambda_t = \frac{ d \phi_t}{dt} $
denotes the (expected) jump rate of the limit system at time $t.$ Then at equilibrium, the total jump rate is constant and given by $ \bar \lambda_t =\bar \lambda.$
From \eqref{eq:limitlambda} we get that
$$ \bar \lambda   = \kappa  \int_0^\infty \psi ( x^{**} , a ) e^{ - \int_0^a \psi (x^{**} , z ) dz } da = \kappa  , $$
where we have used the change of variables $ y = \int_0^a \psi ( x^{**} , z) dz, dy = \psi ( x^{**} , a ) da .$

As a consequence,
$$ \bar \lambda = \kappa  $$
implying that at equilibrium, the jump rate of the system is solution of
\begin{equation}\label{eq:fixedpoint}
\bar \lambda^{-1} = \int_0^\infty \exp \left( { -\int_0^a  \psi \left(\frac{b}{\nu^{n+1} }\bar \lambda , z\right) dz}\right) da .
\end{equation}
Here we have used that at equilibrium
$$ x^{**} = x^{(0 )} =  \frac{1}{\nu } x^{(1)} = \ldots = \frac{b}{\nu^{n+1}} \bar \lambda ,$$
which follows from \eqref{eq:cascade}.

\begin{prop}
Suppose that $ x \mapsto \psi ( x, a ) $ is strictly increasing for any fixed $a \geq 0 $ and that $ b < 0 .$ There exists a unique solution $ \lambda^* $ to \eqref{eq:fixedpoint}.
\end{prop}

Recall that we suppose that $ \psi (x, a ) = f(x) \ql { \{a \geq
  \delta \}},$ for some $\delta > 0 .$ We calculate the right hand side of \eqref{eq:fixedpoint} and obtain the fixed point equation
\begin{equation}\label{eq:fp}
   \int_0^\infty \exp \left( { -\int_0^a  \psi \left(\frac{b}{\nu^{n+1} }\bar \lambda , z\right) dz}\right) da  = \delta + (f( \frac{b}{\nu^{n+1} }\bar \lambda ))^{-1}= \bar \lambda^{-1}.
\end{equation}

More generally, for any Hawkes process with mean-field interactions,
rate function $ \psi (x, a ) $ given by $ \psi ( x, a) = f(x) \ql { \{a \geq \delta \}}$ and  general weight function $ h \in {\mathcal L}^1 ( \R_+) , $ we obtain the fixed point equation
\begin{equation}\label{eq:fixedpoint2}
\bar \lambda^{-1} =  \delta +(f( \bar \lambda  \int_0^\infty h (t) dt ))^{-1}
\end{equation}
for the limit intensity. This limit intensity depends on the length of the refractory period, we write $ \bar \lambda = \bar \lambda ( \delta ) $ to indicate this dependence.

It is then natural to study the influence of the length of the refractory period $ \delta $ on the limit intensity. If $ f $ is increasing and the system inhibitory, that is, $ \int_0^\infty h (t) dt < 0 , $ then clearly
$$ \delta \mapsto \bar \lambda ( \delta ) $$
is decreasing: increasing the length of the refractory period ``calms down the system''.

In the excitatory case when $ \int_0^\infty h(t) dt > 0 ,$ to ensure that the fixed point equation \eqref{eq:fixedpoint} has a solution, suppose that $f$ is strictly increasing and bounded from above and below, away from zero.  Then the function
$$  \bar \lambda \mapsto ( f( \bar \lambda  \int_0^\infty h (t) dt ))^{-1}$$
is a strictly decreasing function mapping $ [0, \infty ] $ onto $ [ \frac{1}{ f (0) } , \frac{1}{f( \infty ) }].$ Therefore, there is exactly one fixed point solution of \eqref{eq:fixedpoint},  and $ \delta \mapsto \bar \lambda ( \delta ) $ is again decreasing.

\section{Appendix}

Here we prove Lemma \ref{odelm}, \Cref{theo:prop} and Lemma 3.4. Then we collect some
useful results about counting processes.

\subsection{Proofs}

\begin{proof}[Proof of Lemma \ref{odelm}]
	It suffices to show that a unique solution exists on $\lf 0,T
        \rf$, for arbitrary $T\geq 0$. In the following proof,
        $C:=C_{T}$ will denote a dynamic constant depending on the
        parameters described in the lemma. It need not represent the
        same constant from line to line, nor from equation to
        equation.

	First we prove existence of a solution to \eqref{diff} with
        $\phi_{t}=\alpha_{t}$ using Picard-iteration.  For $n\in \N$
        define $\lp \phi^{n},x^{n},A^{n}\rp = \lp
        \phi^{n,k},x^{n,k},A^{n,k} \rp_{k\leq \ccK}$ as follows. Initialize
        the system for  $n=0$ by putting $\lp \phi^{0,k},x^{0,k},A^{0,k}
        \rp \equiv   \lp 0,0,A_{0} \rp $. For general  $n\in\N , n
        \geq 1 , $ the triple $\lp \phi^{n},x^{n},A^{n}\rp $ is
        defined as  the solution to \eqref{diff} with
        $\alpha=\phi^{n-1}$ . Inductively it is seen that these
        processes are well-defined. Recall that $h=\sum_{k,l=1}^{\ccK}
        \lv h_{kl} \rv .$  Using \eqref{psias2} we bound $x^{n}$ by
	$$
	\lv x^{n}_{t}\rv \leq \sum_{l=1}^{\ccK} \int^{t}_{0} \lv h\lp t-s \rp \rv d \phi^{n-1,l}_{s} + | \beta_t| \leq  C\int_{0}^{t} \lv h\lp s \rp \rv ds + C\int^{t}_{0} \lv h\lp t-s \rp \rv \lv x^{n-1}_{s}\rv ds + | \beta_t| .
	$$
	It follows from \Cref{convobound} in the Appendix that there
        exists a constant $C>0$ which bounds all $\lV
        x^{n}\rV_{T},n\in \N$. Using this upper bound on $x^{n}, $ we
        also bound the difference of two consecutive solutions. Define
		\begin{align*}
		\delta^{n}_{t}=\sum_{k=1}^{\ccK}\int_{0}^{t} {\mathbb E} \lv\psi^{k}\lp x^{n,k}_{s},A^{n,k}_{s} \rp-\psi^{k}\lp x^{n-1,k}_{s},A_{s}^{n-1,k} \rp\rv ds .
		\end{align*}
The Lipschitz property of $\psi$ and the bound on $x^{n}$ yield
$$
\delta^{n+1}_{t}\leq C\int_{0}^{t} \left( \lv x^{n+1}_{s}-x^{n}_{s}\rv+\sum_{k=1}^{\ccK}P\lp \lV A^{n+1,k}-A^{n,k}\rV_{s}>0 \rp \right) ds.
$$
For the probability term, we note that a necessity for the age processes to differ, is that one of their corresponding intensities catches a $\pi$-singularity which the other one does not catch. This leads to the inequality
\begin{align}
&\sum_{k=1}^{\ccK}P\lp \lV A^{n+1, k}-A^{n,k }\rV_{t}>0 \rp
\\\leq& \sum_{k=1}^{\ccK} P\lp \int_{0}^{t}\int_{0}^{\infty}\lv\ql\lb z\leq \psi^{k}\lp x^{n+1,k}_{s},A^{n+1,k}_{s} \rp\rb-\ql \lb z\leq \psi^{k}\lp x^{n,k}_{s},A^{n,k}_{s}\rp  \rb\rv \pi^{k} \lp ds,dz\rp \geq 1 \rp\nonumber \\
\leq& \, \delta^{n+1}_{t},
\end{align}
where the latter inequality follows by the Markov inequality. By Gronwall's inequality we obtain
\begin{align}
\delta^{n+1}_{t}\leq C \int_{0}^{t} \lv x^{n+1}_{s}-x^{n}_{s}\rv\;ds.\label{eq1}
\end{align}
Moreover, Lemma 22 of \cite{dfh} gives
\begin{align}
\int_{0}^{t}\lv x^{n+1}_{s}-x^{n}_{s}\rv ds \leq \sum_{l=1}^{\ccK}\int_{0}^{t}\int_{0}^{s}  h ( s-u)    \lv d\lp \phi^{n,l}_{u}-\phi^{n-1,l}_{u}\rp \rv    ds \leq \int_{0}^{t} h( t-s) \delta^{n}_{s} ds. \label{eq2}
\end{align}
It therefore follows from \Cref{convobound} in the Appendix that for all $1\le  k\leq K,$
$$
\sum_{n=1}^{\infty} \sup_{ t \le T } |
\phi^{n+1,k}_{t}-\phi^{n,k}_{t}| \, \leq \, \sum_{n=1}^{\infty}\delta^{n}_{T}<\infty  .$$
Thus, $\phi^{n} $ and therefore also $ x^{n}$ converge
locally-uniformly to some $\phi,x$, respectively. Moreover,
\begin{align*}
P\lp A^{n}_{s\leq T}\neq A^{n+1}_{s\leq T} \; i.o.\rp= P\left(\bigcap_{m\in\N} \bigcup_{n\geq m} \left\{ \lV A^{n}-A^{n+1} \; \rV_{T}>0\right\}\right)\leq \lim_{m\con \infty}\sum_{n\geq m}^{\infty}\delta^{n+1}_{T}=0.
\end{align*}

It follows that almost surely, $A^{n}$ converges to some limit $A$ after finitely many iterations.

We need to show that the limit triple $\lp \phi,x,A \rp$  satisfies \eqref{diff} with $ \phi_t = \alpha_t .$ Recall that $x\mapsto \psi^{k}\lp x,a\rp$ is continuous for fixed $a\in \R_{+}$. Since $A^n$ reaches its limit in finitely many iterations, and $\psi$ is continuous in $x$ for fixed $a, $ it follows that
$\lim_{n\con \infty}\psi^{k}\lp x^{nk}_{s},A^{nk}_{s}\rp$ exists for all $s\leq T, $ almost surely. By dominated convergence and \eqref{psias2},  it follows that ${\mathbb E} \psi^{k}\lp x^{n,k}_{s},A^{n,k}_{s}\rp$ converges as well. Therefore, once again by dominated convergence,
$$ \phi_t^k = \lim_{n \to \infty} \phi_t^{n,k} = \lim_{n \to \infty}  \int_0^t {\mathbb E} \psi^{k}\lp x^{n,k}_{s},A^{n,k}_{s}\rp ds =\int_0^t  {\mathbb E} \psi^{k}\lp x^{k}_{s},A^{k}_{s}\rp ds, $$
that is, $ \phi$ satisfies \eqref{diff}. One shows similarly that
\begin{eqnarray*}
 x_t^k& =& \sum_{l=1}^{\ccK}   \lim_{n \to \infty} \int_0^t h_{kl} (t-s) d \phi_s^{n, l } + \beta_t^k \\
& =& \sum_{l=1}^{\ccK}   \lim_{n \to \infty} \int_0^t h_{kl} (t-s)  {\mathbb E} \psi^l ( x_s^{n l, }, A_s^{n ,l } ) ds  + \beta_t^k\\
& =& \sum_{l=1}^{\ccK}   \int_0^t h_{kl} (t-s)  {\mathbb E} \psi^l ( x_s^{ l }, A_s^{ l } ) ds  + \beta_t^k\\
&=& \sum_{l=1}^{\ccK}   \int_0^t h_{kl} (t-s) d \phi_s^{ l } + \beta_t^k,
\end{eqnarray*}
and $x$ satisfies \eqref{diff} as well. For the age process, notice that the c\`agl\`ad process
\begin{align*}
\ve\lp t\rp=\sum_{k=1}^{\ccK}\int_{0}^{t-}\int_{0}^{\infty} \ql\lb z =  \psi^{k}\lp x^{k}_{s},A^{k}_{s}\rp\rb \pi^{k}\lp ds,dz\rp
\end{align*}
has a compensator which is equal to zero for all $ t \geq 0, $ almost surely, by \Cref{compensator}.  Therefore,  $\ve_t = 0 $ for all $ t \geq 0 $ almost surely. This implies that with probability $1,$ $\ql \lb  z\leq \psi^{k} \lp x^{nk}_{s},A_{s}^{k}\rp \rb$ converges $\pi^{k}-$a.e.\ to $\ql \lb  z\leq \psi^{k} \lp x^{k}_{s},A_{s}^{k}\rp \rb$ for all $k\leq \ccK$. As a consequence,
\begin{align*}
A^{k}_{t}-A^{k}_{0} &=t-\lim_{n\con \infty}\int_{0}^{t-} A^{nk}_{s}\ql\lb z\leq \psi^{k}\lp x^{n,k}_{s},A^{nnk}_{s}\rp\rb \pi^{k}\lp ds,dz\rp
\\&=t-\lim_{n\con \infty}\int_{0}^{t-} A^{k}_{s}\ql\lb z\leq \psi^{k}\lp x^{n,k}_{s},A^{k}_{s}\rp\rb \pi^{k}\lp ds,dz\rp
\\&=  t-\int_{0}^{t-}\int_{0}^{\infty} A^{k}_{s}\ql\lb z\leq \psi^{k}\lp x^{k}_{s},A^{k}_{s}\rp\rb \pi^{k}\lp ds,dz\rp ,
\end{align*}
where we have used dominated convergence. Since $x$ is locally bounded, it follows that $\phi$ is $C^{0}.$

To prove uniqueness, we assume that $\lp \tilde{\phi}, \tilde{x} , \tilde{A} \rp$ also solves \eqref{diff} with
$$\tilde{x}^{k}_{t}=\sum_{l=1}^{\ccK}p_{l}\int_{0}^{t}h_{kl}\lp t-s \rp d\tilde{\phi} ^{l}_{s} + \beta_t^k .$$
Define
\begin{align*}
\delta_{t} = \sum_{k=1}^{\ccK}\int_{0}^{t}{\mathbb E} \lv \psi^{k}\lp x^{k}_{s},A^{k}_{s} \rp-\psi^{k} \lp \tilde{x}^{k}_{s},\tilde{A} ^{k}_{s}\rp\rv ds.
\end{align*}
Considerations analogous to the ones given in the proof of existence, gives that
\begin{align*}
\lv x_{t}-\tilde{x}_{t}\rv\leq  \delta_{t}\leq C\int_{0}^{t} h\lp t-s \rp \delta_{s}ds .
\end{align*}
From Gronwall's inequality it follows that $\delta \equiv 0$, and therefore also that $x = \tilde{x} $ on $\lf 0,T\rf.$ From \eqref{diff} it follows immediately  $\phi = \tilde{\phi} $ and $A = \tilde{A} $ almost surely.
\end{proof}

\begin{proof}[Proof  of Lemma \ref{weightapproximation}] Throughout this proof, $C$ is a dynamic constant with dependencies as declared in the theorem. Define the functions $h = \sum_{k,l } \lv h_{kl }\rv , \tilde{h} = \sum_{k,l } \lv \tilde{h}_{kl }\rv $.
 First we prove that the memory processes ${\mathbb E} |X_{t}|,{\mathbb E} |\tilde{X}_{t}| $  are bounded on $[0,T]$ by a suitable constant  $C $. Note that
\begin{eqnarray*}
{\mathbb E} \lv X_t \rv &\leq& \sum_{l=1}^{\ccK}\lp \int_{0}^{t} h\lp t-s\rp{\mathbb E} \psi^{l}\lp X^{l}_{s},A^{l}_{s}\rp ds+{\mathbb E} |R^{l}_{t}| \rp \\
&\leq& C\int_{0}^{t} h\lp t-s\rp {\mathbb E} \lv X_{s}\rv
     ds+C\int_{0}^{t}h\lp s\rp ds +{\mathbb E} |R_{t} | .
\end{eqnarray*}
Since $ {\mathcal E}$ is uniformly integrable, the direct sum  $\lb \sum_{k,l=1 }^{\mathcal{K}}\lv f_{kl}\rv,f_{kl }\in \mathcal{E} \rb$ is uniformly integrable as well. Thus, there exists $b>0$ satisfying
\begin{align}
 \int_{0}^{T} \sum_{k,l = 1 }^{\mathcal{K}} \lv f_{kl} \rv \lp s\rp \ql\lb \sum_{k,l =1 }^{\mathcal{K}}\lv f_{kl}\rv  \lp s\rp>b\rb ds<2^{-1}
\end{align}
for all choices of $\lp f_{kl } \rp\subset {\mathcal E}$. It follows from \Cref{convobound} that $ {\mathbb E} \| X\|_T \le C $ for a suitable  $C$. The same argument shows that also $ {\mathbb E} \| \tilde{X}\|_T \le C .$
Define the total variation measure $\delta_t = \sum_{k=1}^{\ccK} {\mathbb E} \lv d\lp Z^{k1}_{t}-\tilde{Z}^{k1}_{t} \rp	\rv$. We may write
	\begin{eqnarray*}
	\delta_{t}&\leq &{\mathbb E}\sum_{k=1}^{\ccK}\int_{0}^{t} \lv \psi^{k}\lp \tilde{X}^{k}_{s},\tilde{A}^{k1}_{s}\rp-\psi^{k}\lp X^{k}_{s},A^{k1}_{s}\rp \rv ds\nonumber\\
	&\leq &C\sum_{k=1}^{\ccK} \int_{0}^{t}  {\mathbb E}\lv \tilde{X}^{k}_{s}-X^{k}_{s} \rv+P\lp \lV \tilde{A}^{k1}-A^{k1}\rV_{s}>0\rp ds\label{eqbom}.
	\end{eqnarray*}
	As in the proof of \Cref{odelm} we apply Markov's inequality to achieve
	\begin{align*}
	\sum_{k=1}^{\ccK}P\lp \lV \tilde{A}^{k1}-A^{k1}\rV_{t}> 0\rp\leq \delta^{n}_{t}.
	\end{align*}
	We insert this inequality into \eqref{eqbom} to get
	\begin{align}
	\delta_{t}\leq  C\lp \int_{0}^{t} {\mathbb E}\lv \tilde{X}_{s}-X_{s} \rv ds+\int_{0}^{t}\delta_{s} ds \rp.\label{normC2}
	\end{align}
	We now wish to bound the difference of the memory processes. First, define $\gamma = \sum_{k,l =1}^{\mathcal{K}}\lv h_{kl }-\tilde{h}_{kl }\rv$, and note that for any fixed  $k,l,j$ we have for any  $s\geq 0 $

\begin{eqnarray}
	&&\lv \int_{0 }^{s- }h_{kl }\lp s-u \rp dZ^{lj }_{u} - \int_{0 }^{s- }\tilde{h}_{kl }\lp s-u \rp d\tilde{Z}^{lj }_{u}\rv\nonumber\\
&&\leq \int_{0 }^{s- }\lv h_{kl }-\tilde{h}_{kl }\rv\lp s-u \rp dZ^{lj }_{u} + \int_{0 }^{s- } \lv\tilde{h}_{kl }\rv \lp s-u \rp  \lv d\lp Z^{lj }_{u}-\tilde{Z}^{lj }_{u} \rp\rv\nonumber
\\&&\leq \int_{0 }^{s- }\gamma\lp s-u \rp d\sum_{l=1}^{\mathcal{K} } Z^{lj }_{u} + \int_{0 }^{s- } \tilde{h} \lp s-u \rp  \lv d \sum_{l=1}^{\mathcal{K}}\lp Z^{lj }_{u}-\tilde{Z}^{lj }_{u} \rp\rv.\label{mixingintegral}
\end{eqnarray}
We take expectation and apply Lemma 22 of \cite{dfh} to obtain
\begin{multline*}
\mathbb{E}\int_{0 }^{t}\lv \int_{0 }^{s- }h_{kl }\lp s-u \rp dZ^{lj }_{u} - \int_{0 }^{s- }\tilde{h}_{kl }\lp s-u \rp d\tilde{Z}^{lj }_{u}\rv ds\leq \\
\int_{0 }^{t }\gamma\lp t-s \rp L\lp 1+\mathbb{E}\lV X_{T }\rV \rp ds + \int_{0 }^{t } \tilde{h} \lp t-s \rp \delta_{s } ds.
\end{multline*}
Note that this expression does not depend on  $k,l$ nor  $j$. Thus we get
	\begin{eqnarray}
  &&\sum_{k=1}^{\ccK}\int_{0}^{t} {\mathbb E}\lv \tilde{X}^{k}_{s}-X^{k}_{s} \rv ds \nonumber \\
  &&\leq\sum_{k=1}^{\ccK}\int_{0 }^{t }N^{-1 }\sum_{l=1}^{\ccK}\sum_{j = 1 }^{N_{l }}{\mathbb E}\lv \int_{0 }^{s- }h_{kl }\lp s-u \rp dZ^{lj }_{u} - \int_{0 }^{s-}\tilde{h}_{kl }\lp s-u \rp d\tilde{Z}^{lj }_{u}\rv ds \nonumber\\
	&&\leq C\lp \int_{0 }^{t }  \gamma\lp s \rp ds  + \int_{0}^{t} \lv \tilde{h}\lp t-s \rp\rv \delta_{s} ds \rp.\label{h-inequality2}
	\end{eqnarray}
	Inserting inequality \eqref{h-inequality2} into \eqref{normC2}, we obtain
$$
	\delta_{t}\leq  C\lp \int_{0}^{t}\gamma\lp s \rp ds  + \int_{0}^{t} \lp \tilde{h}\lv \lp t-s \rp\rv +1\rp \delta_{s} ds\rp.
$$
  The proof will be complete, after repeating the argument for bounded  ${\mathbb E} \lv X\rv,$ but with  $\delta $ in place of ${\mathbb E} \lv X\rv $.

\end{proof}

\begin{proof}[Proof of \Cref{theo:prop}]
	Let $( \tilde{Z}^{N},\tilde{X}^{N},\tilde{A}^{N})$ be the $N-$dimensional age dependent Hawkes process induced by the same parameters as $(Z^N, X^N, A^N),$ except the weight functions $h_{kl}$ instead of $h^N_{kl}.$

	Fix $ T > 0 $ and consider $t \in [0, T ].$ We have
	\begin{align*}
	\sum_{k=1}^{\ccK}\lv d\lp Z^{Nk1}_{t}-Z^{k1}_{t} \rp	\rv\leq \sum_{k=1}^{\ccK}{\mathbb E} \lv d\lp Z^{Nk1}_{t}-\tilde{Z}^{Nk1}_{t} \rp	\rv+\sum_{k=1}^{\ccK}{\mathbb E} \lv d\lp \tilde{Z}^{Nk1}_{t}-Z^{k1}_{t} \rp	\rv := \tilde{\delta}_t^{Nk}+\delta_t^{Nk}.
	\end{align*}
	The first term converges by \Cref{weightapproximation}, and so it remains to prove convergence of $\delta^{N}_t.$ This part of the proof follows closely the proof given by Chevallier in \cite{chevallier}, but we include it here for completeness. Let $C$ be a dynamic constant depending on $p_{k },L,T,\ccK,\lV r\rV_{T}$ and $\lp h_{kl}\rp$. We use the symbol   $\ve\lp N\rp$ for any function depending on the same parameters as  $C $, and $ N $ such that $\ve \lp N \rp \conx{N\con \infty}0 $.
	Recall that $\lV x\rV_{T}$ is bounded by $C$ sufficiently large by \Cref{odelm}.

	As in the proof of \Cref{odelm} we obtain
	\begin{align}
	\delta^{N}_{t}\leq  C\lp \int_{0}^{t} {\mathbb E}\lv \tilde{X}^{N}_{s}-x_{s} \rv ds+\int_{0}^{t}\delta^{N}_{s} ds \rp.\label{normC}
	\end{align}
	This inequality prepares for an application of Gronwall's inequality, but first we bound $\int_0^t {\mathbb E}\lv \tilde{X}_{s}-x_{s}\rv ds $ using $\delta^{N}_{t}$ as well. Indeed, set $\Lambda^{kj}_{t}:=\int_{0}^{t-} \psi\lp x^{k}_{s},A^{kj}_{s}\rp ds$, which is the compensator of $Z^{kj}$. We write $p_{k}=N_{k}/N+\ve\lp N\rp$  and obtain
	$$
	x^{k}_{t}=N^{-1}\sum_{l=1}^{\ccK}\sum_{j=1}^{N_{l}}\int_{0}^{t-}h_{kl}\lp t-s\rp d\phi^{l}_{s}+r^{k}_{t}+\ve\lp N\rp\sum_{l=1}^{\ccK}\int_{0}^{t} h_{kl}\lp t-s\rp d\phi^{l}_{s}.
	$$
	Since $d \phi^{l}_{s} ={\mathbb E} \psi\lp
	x^{l}_{s},A_{s}^{lj}\rp ds$ and ${\mathbb E} \psi\lp
	x^{l}_{s},A_{s}^{lj}\rp$ are locally bounded, the entire right term  may be replaced by an $\ve$-function.  For fixed $k\leq \ccK ,$ we apply the triangle inequality
	\begin{eqnarray}
	&&\int_{0}^{t}\left( {\mathbb E}\lv \tilde{X}^{Nk}_{s}-x^{k}_{s} \rv-{\mathbb E}\lv  R^{Nk}_{s}-r^{k}_{s}\rv \right) ds\nonumber\leq \ve\lp N\rp\\
	&&+ \int_{0}^{t} N^{-1}\sum_{l=1}^{\ccK}{\mathbb E}\lv \sum_{j=1}^{N_{l}}\int_{0}^{s-}h_{kl}\lp s-u \rp d\lp \phi^{l}_{u}-\Lambda^{lj}_{u}\rp \rv ds \label{b1}\\
	\;\;\;\;&& + \int_{0}^{t}N^{-1}\sum_{l=1}^{\ccK}{\mathbb E}\lv \sum_{j=1}^{N_{l}}\int_{0}^{s-}h_{kl}\lp s-u \rp d\lp \Lambda^{lj}_{u}-Z^{lj}_{u} \rp \rv ds \label{b2}\\
	\;\; \;\;&& +\int_{0}^{t} N^{-1}\sum_{l=1}^{\ccK} {\mathbb E}\lv \sum_{j=1}^{N_{l}}\int_{0}^{s-}h_{kl}\lp s-u \rp d\lp Z^{lj}_{u}- \tilde{Z}^{Nlj}_{u} \rp  \rv ds\\
	&&:=  \ve\lp N\rp+B^{1k}_{t}+B^{2k}_{t}+B^{3k}_{t}.\nonumber
	\end{eqnarray}
	We now proceed to bound $B^{i}:=\sum_{k=1}^{\ccK}B^{ik}$, $i\leq 3$. Define $h=\sum_{k,l=1}^{\ccK}\lv h_{kl}\rv$. Rewrite  $\phi$ and $\Lambda$ in terms of their densities,  and thereby obtain a bound for the inner-most sum in \eqref{b1} for $s\in \lf 0,t \rf,l\leq \ccK, $ which is given by
	$$
	{\mathbb E}\int_{0}^{s} \sum_{j=1}^{N_{l}}   h\lp s-u \rp  |d ( \phi^{l}_{u}-\Lambda^{lj}_{u})  |
	\leq \int_{0}^{s}  h\lp s-u \rp  {\mathbb E} \sum_{j=1}^{N_{l}}\lv \psi^{l}\lp x^{l}_{u},A^{lj}_{u} \rp-{\mathbb E}\psi^{l}\lp x^{l}_{u},A^{lj}_{u} \rp\rv du.
	$$
	Notice that the sum consists of i.i.d.\ terms, so we may apply Cauchy-Schwarz to bound it by $\sqrt{N_{l}\, \text{Var} ( \psi (  x^{l}_{u},A^{lj}_{u})  ) }$, which is bounded for $u\in \lf 0,T \rf$ by $\sqrt{N_{l}} C\lp 1+\lV r^{l}\rV_{T} \rp  $ using \eqref{psias2}. Insert this into \eqref{b1} to see that
	$$
	B_t^{1}=\sum_{k=1}^{\ccK}B^{1k}_{t} \leq\ve\lp N\rp.
	$$
	For $B_t^{2}$, recall that  $(Z^{lj}-\Lambda^{lj})_j $ are i.i.d.\ for fixed $l.$ By Cauchy-Schwarz, we obtain a bound for the inner-most sum of \eqref{b2}
	\begin{equation}
	N_{l}^{1/2}\sqrt{\text{Var}\int_{0}^{s}h_{kl}\lp s-u \rp d\lp Z^{l1}_{u}-\Lambda^{l1}_{u}\rp }\; . \label{B2}
	\end{equation}
	To treat the process inside the root, fix $s\geq 0,l\leq \ccK$ and consider the process
	$$
	I:r\mapsto \int_{0}^{r\wedge s}h_{kl}\lp s-u \rp d\lp Z^{l1}_{u}-\Lambda^{l1}_{u}\rp.
	$$
	Then $I$ is a martingale, and
	$$
	\text{Var }I_{s}={\mathbb E}[I]_{s}={\mathbb E}\int_{0}^{s} h_{kl}^2\lp s-u\rp d\Lambda^{l1}_{u}=\int_{0}^{s} h_{kl}^2\lp s-u\rp {\mathbb E} \psi^{l} \lp x^{l}_{u},A^{l1}_{u} \rp du.
	$$
	Since $h_{kl}\in {\mathcal L}^{2}_{loc}$ it follows that $\text{Var} \;I_{s}$ is bounded on $s\in [0,T], $ and so
	$$
	B^{2}_{t}\leq \ve\lp N\rp
	$$
	for all $ t \le T. $
	For $B^{3}$ the triangle inequality, and Lemma 22 \cite{dfh} gives
	\begin{align*}
	B^{3}_{t}\leq C\int_{0}^t  h\lp t-s \rp   \delta_{s}^{N} ds .
	\end{align*}
	We plug the bounds for  $B_{1},B_{2}$ and  $B_{3}$ into \eqref{normC} to obtain
	\begin{align*}
	\delta^{N}_{t}\leq C\lp \int_{0}^{t}  h\lp t-s \rp \delta^{N}_{s}+ \ve\lp N\rp+\sum_{k=1}^{\ccK}{\mathbb E} \lv R^{Nk}_{s}-r^{k}_{s}\rv ds \rp.
	\end{align*}
	Applying \Cref{convobound} in the Appendix yields
	\begin{align}
	\delta^{N}_{t}\leq  \ve\lp N\rp+C\int_{0}^T {\mathbb E} \lv R_{s}^N-r_{s}\rv ds = \ve\lp N\rp
	\end{align}
	for all $t\leq T$, which implies the desired result.

\end{proof}

\subsection{Results about counting processes}

\subsubsection{Measure theory on $\lp M_{E},{\mathcal M}_E \rp$ }
This section provides a brief overview of measure theory on the measurable space $\lp M_{E},{\mathcal M}_E \rp$ of bounded measures defined on a Polish space $E.$ We refer to  A.2.6 in \cite{DALEY} for more details.
Let $d$ be a distance so that $\lp E,d\rp$ is complete and separable. A measure $\nu$ on $E$ is said to be \textit{boundedly finite} if $\text{diam}\lp A\rp<\infty $ implies that $  \nu ( A) <\infty$ for $A\in {\mathcal B}\lp E\rp$. Let $M_{E}$ be the space of all boundedly-finite measures on $(E, {\mathcal B} (E) )$. This space may now be equipped with the \textit{weak-hat} metric $\hat{d}$, making $( M_{E},\hat{d})$  a complete and separable space itself. The Borel-algebra ${\mathcal M}_{E}$ is easily characterized by the projections  $\Pi_{A}:M_{E}\ni  \nu \mapsto \nu\lp A\rp$ in the sense that
\begin{align}
{\mathcal M}_{E} =\sigma\lp \Pi_{A},   \;  A\in {\mathbb D} \rp , \label{generate}
\end{align}
where ${\mathbb D}\subset {\mathcal B}\lp D\rp$ is a semi-ring of bounded sets. A random variable taking values in the space of measures is called a \textit{random measure}. A particular interesting example is when $E=\R \times \R_+ $ as considered in this paper. A Poisson random measure  $\pi:\Omega\rightarrow M_{\R\times\R_{+}}$ on $ \R\times \R_+ $ with Lebesgue intensity measure is a random measure such that for any disjoint $A_{1},\dots ,A_{n}\in {\mathcal B} ( \R ) \times {\mathcal B} ( \R_+) , $
\begin{itemize}
	\item $\pi\lp A_{i}\rp\sim \text{Pois}\lp \int_{A_{i}}dsdz \rp$
	\item $\pi\lp A_{1}\rp\indep\dots \indep \pi\lp A_{n}\rp .$
\end{itemize}
Since the underlying space is on the form $\R \times \R_+, $ the first coordinate of $\pi$ may be thought of as the time coordinate;  and concepts like stationarity and ergodicity transfer naturally to random measures. Define the shift operator as the automorphism on $M_{\R\times \R_+}$ given by
\begin{align*}
\lp \theta^{r}\nu\rp\lp C\rp = \nu \lp \lb \lp t,x\rp\in \R\times \R_+  : \lp t- r,x \rp\in C \rb\rp .
\end{align*}
Then a random measure $\sigma$ on $\R\times \R_+ $ is said to be \textit{stationary} if the distribution is invariant under shift
\begin{align*}
{\mathcal L} ( \sigma) = {\mathcal L} \lp \theta^{r}\sigma\rp
\end{align*}
for all $ r \in \R.$ Stationarity is equivalent to have invariance of the finite dimensional distributions (fidi's)  (Proposition 6.2.III of \cite{DALEY})
\begin{align*}
	P\left(\bigcap_{i=1}^{n}\left\{\sigma\lp A_{i}\rp\in B_{i}\right\}\right)=P\left(\bigcap_{i=1}^{n} \left\{ \theta^{r}\sigma\lp A_{i}\rp\in B_{i}\right\}\right)
	\end{align*}
for all $r\in \R,n>0,A_{1},\dots , A_{n}\in {\mathcal B} ( \R \times \R_+) ,B_{1},\dots,B_{n}\in {\mathcal B} ( \R_+ ) .$ A stationary random measure $\sigma$ is \textit{mixing} if
\begin{align}\label{eq:mix}
P\lp \sigma \in V,  \theta^r \sigma\in W\rp\conx{\lv r \rv \con \infty} P\lp \sigma\in V\rp P\lp \sigma \in W\rp
\end{align}
for all $ V, W \in {\mathcal M}_{\R \times \R_+}.$ We refer to Chapter 10.2-10.3 of \cite{DALEY} for a thorough introduction to ergodic theory for random measures.  As is the case for processes, mixing implies that ${\mathcal L} (\sigma) $ is ergodic w.r.t.\ the shift operator $\theta^{r}$ for all $r>0$, meaning that all invariant events have probability $0/1$.

Finally, we present a core measurability result, which can be applied to show measurability of all processes treated in this article.
\begin{lm}
	\label{measurelemma}
	Let $ D,E$ be complete and separable metric spaces, and let $H:D\times E\con \R_+ $ be measurable. The section integral $F: M_E  \times D \con \overline{\R} $
	\begin{align*}
	F\lp \nu, d \rp=\int_E H\lp d ,s\rp d\nu\lp s\rp
	\end{align*}
	is ${\mathcal M}_{E} \times {\mathcal B} ({D} ) \con {{\mathcal B}} ( \overline \R) $ measurable.
\end{lm}
\begin{proof}
We start by defining $G:  M_{ D\times E} \con \overline{\R_+}$ as
\begin{align*}
G:\rho \mapsto \int H\lp x,y\rp \rho (dx, dy) .
\end{align*}
It is easily seen that $G$ is measurable.

Consider the map $m : M_E \times D \to M_{D \times E } $ given by $ \lp \nu , d\rp\mapsto \delta_{d} \otimes \nu  $. To prove measurability of $m$ it is sufficient to treat projections into bounded boxes $A\times B,  A\in {\mathcal B} ({D}) ,B\in {\mathcal B} ({E}) $. Such projections are simply given as $\Pi_{A\times B} m : \lp \nu,d  \rp \mapsto \ql \lb B\rb\lp d \rp\Pi_{A}\lp \nu \rp$ and are therefore measurable. We conclude that
\begin{align*}
G\circ m \lp \nu,d \rp = \int\int H\lp u,s \rp\; d\delta_{d}\lp u\rp d\nu\lp s\rp =\int H\lp d,s\rp d\nu\lp s\rp=F\lp \nu,d\rp ,
\end{align*}
proving that $F$ is measurable.
\end{proof}

We shall need the following version of Gronwall's lemma which has been proven in \cite{dfh}. Recall that for any function $g  :  \R_{+}\con \R_{}$ and any $ T > 0, $ we have introduced $\|g\|_T = \sup_{ t \le T } | g(t) |.$
\begin{lm}[Lemma 23 of \cite{dfh}]\label{convobound} Let $h : \R_{+}\con \R_{+}$ be locally integrable and $g  :  \R_{+}\con \R_{+}$ be locally bounded. Let $T\geq 0.$\\
1. Let $u$ be a locally bounded nonnegative function satisfying $u_{t}\leq g_{t}+\int_{0}^{t}h\lp t-s\rp u_{s} ds$ for all $t\in [0,T]$. If $b> 0$ satisfies that
	\begin{align}
	 \int_{0}^{T}h\lp s\rp \ql\lb h\lp s\rp \geq b\rb ds<\frac12 ,  \label{A-equality}
	\end{align} then $\lV u \rV_{T}\leq 2e^{2bT}\lV g\rV_{T} =: C_T \lV g\rV_{T} .$\\
2.  Let $\lp u^{n}\rp$ be a sequence of locally bounded nonnegative functions such that
	$u^{n+1}_{t}\leq g_{t}+ \int_{0}^{t}h\lp t-s\rp u^{n}_{s}ds$ for all $t\in [0,T]$. Then $\sup_{n}\lV u^{n}\rV_{T}\leq C_{T}\lp \lV g \rV_{T}+\lV u^{0}\rV_{T}\rp.$ Moreover,  if the inequality is satisfied with $g\equiv 0, $ then $\sum_{n} u^{n}$ converges uniformly on $[0,T]$.
\end{lm}

\subsubsection{Point Process Results}
We collect some useful results on point processes known in the literature.

\begin{lm}
	\label{compensator}

	Let $H:\lp \Omega\times \R\times \R\rp \con \R$ be ${\mathcal P}\otimes {\mathcal B}\rightarrow {\mathcal B}$ measurable and assume that almost surely
	\begin{align*}
	Z : t\mapsto \int_{0}^{t}\int_{0}^{\infty} H\lp s,z\rp \pi\lp ds,dz\rp
	\end{align*}
	does not explode; that is, for all $ t > 0, $
	$$ \int_{0}^{t}\int_{0}^{\infty} |H\lp s,z\rp |\pi\lp ds,dz\rp < \infty $$ almost surely.
	\begin{enumerate}
	\item If $ H $ is bounded, then the compensator  $\Lambda$ of $Z$ is given by
	\begin{align*}
	\Lambda : t\mapsto \int_{0}^{t}\int_{0}^{\infty} H\lp s,z\rp  dz\; ds ,
	\end{align*}
	i.e.\ $Z-\Lambda$ is a local $\lp {\mathcal F}_{t}\rp$-martingale.
	\item If moreover  $s\mapsto {\mathbb E} \int | H\lp s,z\rp | dz$ is locally integrable, then $Z-\Lambda$ is a martingale.
	\item Fix $T\geq 0$ and assume that $\Lambda$ can be written as
	\begin{align*}
	\Lambda_{t}=\int_{0}^{t} \lambda_{s} ds.
	\end{align*}

	Assume also that $\lambda\lp s\rp=F\lp Z_{\vert \lp -\infty,T\rp },s\rp+\ve\lp s\rp$ where $F$ is $\mathcal{M}_{\R} \times {\mathcal B} ( \R ) \mapsto {\mathcal B} ( \R) $ measurable and $t\mapsto \ve\lp t\rp$ is $\lp {\mathcal F}_{t\wedge T}\rp$-predictable.  It holds that
	\begin{align*}
	P\lp Z\lp T,\infty\rp=0 \; \vert {\mathcal F}_{T}\rp \stackrel{a.s.}{= } \exp\lp -\int_{T}^{\infty} F\lp Z_{\vert \lp -\infty,T\rf },s\rp+\ve\lp s\rp ds\rp.
	\end{align*}

	\end{enumerate}
\end{lm}
\begin{proof}
The first point follows from  \cite{jacod}, Theorem 1.8 of Chapter II,  by using the localizing sequence $ T_n = \inf \{ t :  \int_{0}^{t}\int_{0}^{\infty} |H\lp s,z\rp |\pi\lp ds,dz\rp \geq n \} , n \geq 1 ,$ since
$$   \int_0^{T_n}\int_{0}^{\infty} |H\lp s,z\rp |\pi\lp ds,dz\rp  \le n  + \| H \|_\infty .$$
For the second point, let $M_t := Z_t - \Lambda_t .$ It suffices to show that $ \E ( \sup_{s \le t } | M_s|) < \infty  $ which follows from
$$ \E \left(\sup_{ s \le t } | M_{ s \wedge T_n } |\right) \le 2 \E \int_0^t \int_0^\infty | H ( s, z ) | ds dz < \infty $$
by monotone convergence. The third point is Lemma 1 in \cite{bm}.
\end{proof}

\begin{lm}[Lemma 5 of \cite{bm}]
	\label{ergodiclemma}
	Assume that $X,Y$ are $\lp {\mathcal F}_{t}\rp_{t\in \R_+}$-progressive and that for all $ s \geq 0$
	\begin{align*}
	P\lp X_{t} = Y_{t}\;  \forall t>s\; \vert {\mathcal F}_{s} \rp \geq U_{s} -r \lp s\rp ,
	\end{align*}
	where $U$ is ergodic, $P\lp U_{s}>0\rp>0$ and $r \lp t\rp\conx{a.s.}0$ for $ t \to \infty .$ Then almost surely, $X$ and $Y$ couple in finite time.
\end{lm}


\begin{thebibliography}{99}


	\bibitem{Hawkes-In-Finance}
{\sc Bacry, E., Muzy, J.F.}
\newblock Hawkes model for price and trades high-frequency dynamics.
\newblock {\em Quantitative finance}, 14 (2014), 1--20.



	\bibitem{bm}
	{\sc Br\'emaud, P., Massouli\'e, L.}
	\newblock Stability of nonlinear Hawkes processes.
	\newblock {\em The Annals of Probability}, 24(3) (1996) 1563-1588.

	\bibitem{chevallier}
	{\sc Chevallier, J.}
	\newblock Mean-field limit of generalized Hawkes processes.
	\newblock {\em  Stoc. Proc. and their Appl.} 127 (2017), 3870-3912.

	\bibitem{ccdr}
	{\sc Chevallier, J., Caceres, MJ., Doumic, M., Reynaud-Bouret, P.}
	\newblock Microscopic approach of a time elapsed neural model.
	\newblock {\em Math. Mod. \& Meth. Appl. Sci.}, 25(14) (2015) 2669--2719.

\bibitem{costa}
	{\sc Costa, M, Graham, C., Marsalle, L, Tran, V.C.}
	\newblock Renewal in Hawkes processes with self-excitation and inhibition.
	\newblock arxiv.org/abs/1801.04645, 2018.


	\bibitem{DALEY}
	{\sc Daley, D., Vere-Jones, D.}
	\newblock An Introduction to the Theory of Point Processes.
	\newblock {\em Springer-Verlag New York, (1988)}.
	\bibitem{dfh}
	{\sc Delattre, S., Fournier, N., Hoffmann, M.}
	\newblock  Hawkes processes on large networks.
	\newblock {\em Ann. App. Probab.} 26 (2016), 216--261.


	\bibitem{SusEva}
	{\sc Ditlevsen, S., L\"ocherbach, E.}
	\newblock Multi-class oscillating systems of interacting neurons.
	\newblock {\em Stoc. Proc. and their Appl.} 127 (2017), 1840--1869.


\bibitem{aline}
{\sc Duarte, A., L\"ocherbach, E., Ost, G.}
\newblock Stability, convergence to equilibrium and simulation of non-linear Hawkes Processes with memory kernels given by the sum of Erlang kernels.
\newblock Submitted, arxiv.org/abs/1610.03300, 2017.




\bibitem{evafou}
{\sc Fournier, N., L\"ocherbach, E.}
\newblock A toy model of interacting neurons.
\newblock {\em Annales de l'I.H.P.} 52 (2016), 1844--1876.

\bibitem{GerhardDegerTruccolo2017}
{\sc Gerhard, F., Deger, M. and Truccolo, W.}
\newblock {On the stability and dynamics of stochastic spiking neuron models: Nonlinear Hawkes process and point process GLMs}
\newblock {\em PLOS Computational Biology}, 13(2) (2017): e1005390. doi:10.1371/journal.pcbi.1005390.

\bibitem{carl2}
	{\sc Graham, C., Robert, P.}
	\newblock Interacting multi-class transmissions in large stochastic systems.
	\newblock {\em Ann. Appl. Prob. 19}, 6 (2009) 2334--2361.




         \bibitem{Halmos}
{\sc Halmos, P.R.}
\newblock Measure theory.
\newblock {\em Springer-Verlag New York}, (1974).


	\bibitem{Hawkes}
	{\sc Hawkes, A. G.}
	\newblock Spectra of Some Self-Exciting and Mutually Exciting Point Processes.
	\newblock {\em Biometrika}, 58 (1971) 83-90.

	\bibitem{ho}
	{\sc Hawkes, A. G. and Oakes, D.}
	\newblock  A cluster process representation of a self-exciting
	process.
	\newblock {\em J. Appl. Prob.}, 11 (1974)  93-503.


\bibitem{jacod}
{\sc Jacod, J., Shiryaev, A.N.}
\newblock Limit theorems for stochastic processes.
\newblock Second edition, Springer-Verlag Berlin, (2003).


\bibitem{miha}
{\sc Miyahara, Y.}
\newblock Invariant measures of ultimately bounded stochastic processes.
\newblock {\em Nagoya Math. J.}, 49 (1973), 149--153.

\bibitem{nummelin}
{\sc Nummelin, E.}
\newblock  A splitting technique for Harris recurrent Markov chains.
\newblock {\em Z. Wahrscheinlichkeitstheorie Verw. Geb.} 43 (1978), 309--318.

\bibitem{salort}
{\sc Pakdaman, K., Perthame, B., Salort, D.}
\newblock Relaxation and self-sustained oscillations in the time elapsed neuron network model.
\newblock arxiv.org/abs/1109.3014, 2011.

\bibitem{pat}
{\sc Reynaud-Bouret, P.,  Schbath, S.}
\newblock{Adaptive estimation for Hawkes processes; application to genome analysis.}
\newblock{\em The Annals of Statistics}, 5 (2010), 2781--2822.

\bibitem{patvin}
{\sc Reynaud-Bouret, P.,  Rivoirard, V. and Tuleau-Malot, C.}
\newblock{Inference of functional connectivity in Neurosciences via
  Hawkes processes.}
\newblock{1st IEEE Global Conference on Signal and Information Processing, Austin, Texas (2013)}

\bibitem{Social-Media}
{\sc Zadeh, A. H., Sharda, R.}
\newblock Hawkes Point Processes for Social Media Analytics.
\newblock {\em Reshaping Society through Analytics, Collaboration, and Decision Support: Role of Business Intelligence and Social Media}, Springer International Publishing, (2015).




\end{thebibliography}
\end{document}